\documentclass{amsart}

\usepackage{amsmath,amsthm,amssymb}
\usepackage{a4wide}
\usepackage{hyperref}
\usepackage{graphicx,xcolor}
\usepackage{tikz,ytableau}
\usetikzlibrary{decorations.pathreplacing}
\ytableausetup{aligntableaux=bottom}
\numberwithin{equation}{section}

\newtheorem{thm}{Theorem}[section]
\newtheorem{lem}[thm]{Lemma}
\newtheorem{prop}[thm]{Proposition}
\newtheorem{cor}[thm]{Corollary}
\theoremstyle{definition}

\newtheorem{defn}[thm]{Definition}
\newtheorem{notation}[thm]{Notation}

\newcommand\type{\operatorname{type}}
\newcommand\muentry[1]{*(lightgray)\overline{#1}}
\newcommand\VT{\operatorname{VT}}
\newcommand\sh{\operatorname{sh}}
\newcommand\Sym{\mathfrak{S}}
\newcommand\Tab{\operatorname{Tab}}
\newcommand\LL{\mathcal{L}_\lm}
\renewcommand\L{\mathcal{L}}
\newcommand\lsnc{\LL^{\operatorname{SNC}}}
\newcommand\lnc{\LL^{\operatorname{NC}}}
\newcommand\oRSE{\overline{\RSE}}
\newcommand\col{\operatorname{col}}
\newcommand\row{\operatorname{row}}
\newcommand\pd{\phi_+}
\newcommand\pu{\phi_-}
\newcommand\ZZ{\mathbb{Z}}
\newcommand\NN{\mathbb{N}}

\newcommand\RSE{\operatorname{RSE}}
\newcommand\RPP{\operatorname{RPP}}
\newcommand\SSYT{\operatorname{SSYT}}
\newcommand\pp{\mathbf{p}}
\newcommand\sign{\operatorname{sign}}
\newcommand\lm{{\lambda/\mu}}
\renewcommand\gg{\widetilde{g}}
\newcommand\wt{\operatorname{wt}}

\newcommand\RSEgridonly[3]{
  \def\X{#1} \def\Y{#2} \def\W{#3}
  \foreach \i in {0,...,\X}
  {
    \draw[gray,very thin] (\i,0) -- (\i,\Y+\W+4);
  }
  \pgfmathsetmacro{\y}{\W+1}
  \pgfmathsetmacro{\z}{\W}
  \pgfmathsetmacro{\m}{\Y}
  \foreach \y in {0,...,\m}
  {
    \draw[gray,very thin] (0,\y) -- (\X,\y);
  }
  \pgfmathsetmacro{\m}{\W+2}
  \foreach \y in {2,...,\m}
  {
    \draw[gray,very thin] (0,\Y+\y) -- (\X,\Y+\y);
  }
  \draw[gray,very thin] (0,\Y+\W+4) -- (\X,\Y+\W+4);
}
\newcommand\RSEgrid[3]{
  \RSEgridonly{#1}{#2}{#3}
  \def\X{#1} \def\Y{#2} \def\W{#3}
  \node [left] at (0,\Y+2) {$\omega$};
  \node [left] at (0,\Y+\W+4) {$2\omega$};
  \node [left] at (0,\Y+1.3) {$\vdots$};
  \node [left] at (0,\Y+\W+3.3) {$\vdots$};
  \foreach \j in {0,...,\Y}
  {
    \node [left] at (0,\j) {\j};
  }
  \pgfmathsetmacro{\y}{\W+1}
  \pgfmathsetmacro{\z}{\W}
  \foreach \j in {1,...,\z}
  {
    \node [left] at (0,\Y+2+\j) {$\omega$+\j};
  }
}

\title{Jacobi--Trudi formula for refined dual stable Grothendieck polynomials}

\author{Jang Soo Kim}
\thanks{The author was supported by NRF grants \#2019R1F1A1059081 and \#2016R1A5A1008055.} 
\address{Department of Mathematics,
Sungkyunkwan University (SKKU), Suwon, Gyeonggi-do 16419, South Korea}
\email{jangsookim@skku.edu}

\begin{document}

\begin{abstract}
  In 2007 Lam and Pylyavskyy found a combinatorial formula for the dual stable
  Grothendieck polynomials, which are the dual basis of the stable Grothendieck
  polynomials with respect to the Hall inner product. In 2016 Galashin,
  Grinberg, and Liu introduced refined dual stable Grothendieck polynomials by
  putting additional sequence of parameters in the combinatorial formula of Lam
  and Pylyavskyy. Grinberg conjectured a Jacobi--Trudi type formula for refined
  dual stable Grothendieck polynomials. In this paper this conjecture is proved
  by using bijections of Lam and Pylyavskyy.
\end{abstract}

\maketitle

\section{Introduction}

In 1982 Lascoux and Sch\"utzenberger \cite{LS1982} introduced \emph{Grothendieck
  polynomials}, which are representatives of the structure sheaves of the
Schubert varieties in a flag variety. Fomin and Kirillov \cite{FK1996} studied
Grothendieck polynomials combinatorially and introduced \emph{stable
  Grothendieck polynomials}, which are stable limits of Grothendieck
polynomials. Buch \cite{Buch2002} found a combinatorial formula for stable
Grothendieck polynomials using set-valued tableaux. Lam and Pylyavskyy
\cite{LP2007} first studied \emph{dual stable Grothendieck polynomials}
$g_\lambda(x)$, which are the dual basis of the stable Grothendieck polynomials
under the Hall inner product. They also found a combinatorial formula for
$g_\lambda(x)$ in terms of reverse plane partitions. Their formula gives a
combinatorial way to expand $g_\lambda(x)$ in terms of Schur functions
$s_\mu(x)$.

Dual stable Grothendieck polynomials $g_\lambda(x)$ are inhomogeneous symmetric
functions in variables $x=(x_1,x_2,\dots)$. Galashin, Grinberg, and Liu
\cite{GGL2016} introduced \emph{refined dual stable Grothendieck polynomials}
$\gg_{\lm}(x;t)$ by putting an additional sequence $t=(t_1,t_2,\dots)$ of
parameters in the combinatorial formula of Lam and Pylyavskyy. They showed that
$\gg_{\lm}(x;t)$ is also symmetric in $x$. Refined dual stable Grothendieck
polynomials generalize both dual stable Grothendieck polynomials and Schur
functions: if $t_i=1$ for all $i\ge1$, then $\gg_{\lm}(x;t)=g_{\lm}(x)$, and if
$t_i=0$ for all $i\ge1$, then $\gg_{\lm}(x;t)=s_{\lm}(x)$. Galashin
\cite{Galashin2017} found a Littlewood--Richardson rule to expand
$\gg_{\lm}(x;t)$ in terms of Schur functions. Yeliussizov \cite{Yeliussizov2017}
further studied (dual) stable Grothendieck polynomials and showed the following
Jacobi--Trudi type formula for $\gg_{\lambda}(x;t)$ originally conjectured by
Darij Grinberg \cite{Grinberg_conj}:
\begin{equation}
  \label{eq:Yeliussizov}
  \widetilde{g}_{\lambda}(x;t) = \det
  \left( e_{\lambda'_i-i+j}(x_1,x_2,\dots,t_1,t_2,\dots,t_{\lambda'_i-1})
  \right)_{1\le i,j\le n},
\end{equation}
where $e_k(z_1,z_2,\dots)=\sum_{i_1<i_2<\dots<i_k}z_{i_1}z_{i_2}\dots z_{i_k}$
is the $k$th elementary symmetric function and we define $e_0(z_1,z_2,\dots)=1$
and $e_k(z_1,z_2,\dots)=0$ for $k<0$. We also note that the Jacobi--Trudi
formula \eqref{eq:Yeliussizov} for the case $t_i=1$ for all $i\ge1$ was first
proved by Shimozono and Zabrocki \cite{ShimozonoZabrocki}.

The main result of this paper is the following Jacobi--Trudi formula for the
refined dual stable Grothendieck polynomial $\gg_{\lm}(x;t)$, which was also
conjectured by Darij Grinberg \cite[slide 72]{Grinberg_conj} in 2015. See
Section~\ref{sec:determ-expans} for the precise definitions.

\begin{thm}\label{thm:main}
  Let $\lambda$ and $\mu$ be partitions with $\ell(\lambda')\le n$. Then
\[
  \widetilde{g}_{\lambda/\mu}(x;t) = \det
  \left( e_{\lambda'_i-\mu'_j-i+j}
    (x_1,x_2,\dots,t_{\mu'_j+1},t_{\mu'_j+2},\dots,t_{\lambda'_i-1})
  \right)_{1\le i,j\le n},
\]
where, if $\mu'_j+1>\lambda'_i-1$, the $(i,j)$ entry is defined to be
$e_{\lambda'_i-\mu'_j-i+j} (x_1,x_2,\dots)$.
\end{thm}

Note that if $t_i=0$ for all $i\ge1$, then Theorem~\ref{thm:main} reduces to the
classical (dual) Jacobi--Trudi formula for the Schur function $s_\lm(x)$.
Theorem~\ref{thm:main} gives another proof of the fact that $\gg_{\lm}(x;t)$ is
symmetric in the variables $x$. If $t_i=1$ for all $i$, we obtain a new
Jacobi--Trudi formula for the dual stable Grothendieck polynomial $g_\lm(x)$.

\begin{cor}\label{cor:main}
  Let $\lambda$ and $\mu$ be partitions with $\ell(\lambda')\le n$. Then
\[
  g_{\lambda/\mu}(x) = \det
  \left(\sum_{k\ge0} \binom{\lambda_i'-\mu_j'-1}{k} e_{\lambda'_i-\mu'_j-i+j-k}
    (x_1,x_2,\dots) \right)_{1\le i,j\le n},
\]
where we define $\binom{m}{k}=\delta_{k,0}$ if $m<0$.
\end{cor}

There is a standard combinatorial method to prove a Jacobi--Trudi type formula
using the Lindstr\"om--Gessel--Viennot lemma \cite{LGV,Lindstrom}.
First interpret the determinant as a signed sum of $n$-paths, i.e., sequences
$(p_1,\dots,p_n)$ of $n$ paths in a certain lattice. If there are intersections
among the $n$ paths, choose an intersection in a controlled way and exchange the
``tails'' of the two paths through this intersection. This will give a
sign-reversing involution on the total $n$-paths leaving only the
non-intersecting $n$-paths as fixed points. Then one interprets the
non-intersecting $n$-paths as the desired tableaux by a simple bijection.
This is how Yeliussizov \cite{Yeliussizov2017} proved \eqref{eq:Yeliussizov}.

However, the Jacobi--Trudi formula in Theorem~\ref{thm:main} cannot be proved in
this way. Because of the restriction of a path depending on the initial point,
the usual method of exchanging tails is not applicable. In this paper we prove
Theorem~\ref{thm:main} by finding a sign-reversing involution on certain
$n$-paths using two maps introduced by Lam and Pylyavskyy \cite{LP2007} as
intermediate steps in their bijection between reverse plane partitions and pairs
of semistandard Young tableaux and so-called elegant tableaux.

The remainder of this paper is organized as follows. In
Section~\ref{sec:determ-expans} we give necessary definitions and notation, and
show that the determinant in Theorem~\ref{thm:main} is a generating function for
``semi-noncrossing'' $n$-paths. In Section~\ref{sec:vertical-tableaux} we define
vertical tableaux and give a bijection between them and $n$-paths. In
Section~\ref{sec:rse-tableaux} we define RSE-tableaux and review two bijections
$\pu$ and $\pd$ on RSE-tableaux due to Lam and Pylyavskyy. In
Section~\ref{sec:skew-rse-tableaux}, we extend the definition of RSE-tableaux to
skew shapes and study properties of the maps $\pu$ and $\pd$ on skew
RSE-tableaux. In Section~\ref{sec:sign-revers-invol} we give a sign-reversing
involution on semi-noncrossing $n$-paths using these maps and complete the proof
of Theorem~\ref{thm:main}. In Section~\ref{sec:example-phi} we give a concrete
example of the sign-reversing involution defined in
Section~\ref{sec:sign-revers-invol}.

We note that Amanov and Yeliussizov \cite{AmanovYeliussizov} proved
Theorem~\ref{thm:main} and Corollary~\ref{cor:main} independently about the same
time this paper was written. Their proof uses a sign-reversing involution on
3-dimensional lattice paths. It would be interesting to see whether there is a
connection between their sign-reversing involution and ours.

\section{Definitions and notation}
\label{sec:determ-expans}

In this section we give basic definitions and notation which will be used
throughout this paper.

A \emph{partition} $\lambda=(\lambda_1,\lambda_2,\dots,\lambda_\ell)$ is a
weakly decreasing sequence of positive integers. Each $\lambda_i$ is called a
\emph{part} of $\lambda$. The \emph{length} $\ell(\lambda)$ of $\lambda$ is the
number of parts. Sometimes we will append some zeros at the end of $\lambda$ so
that for example $(4,3,1)$ and $(4,3,1,0,0)$ are considered as the same
partition, and $\lambda_i=0$ whenever $i>\ell(\lambda)$.

The \emph{Young diagram} of $\lambda$ is defined to be the set $\{(i,j)\in\ZZ^2:
1\le i\le \ell(\lambda), 1\le j\le \lambda_i\}$. From now on we will identify
$\lambda$ with its Young diagram. The Young diagram of $\lambda$ will be
visualized by placing a unit square, called a \emph{cell}, in the $i$th row and
$j$th column for each $(i,j)\in\lambda$. The \emph{conjugate} $\lambda'$ of
$\lambda$ is defined to be the partition given by
$\lambda'=\{(i,j):(j,i)\in\lambda\}$, see Figure~\ref{fig:YD}.

\begin{figure}
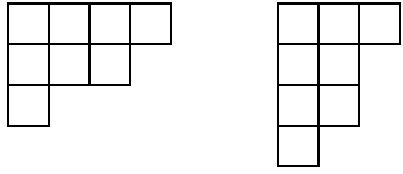

  \centering
  \begin{ytableau}
   {} & & & \\
   {} & &  \\
   {}  \\
   \none  \\
  \end{ytableau}\qquad\qquad
  \begin{ytableau}
   {}  & & \\
   {}  &  \\
   {}  &  \\
   {}  \\
  \end{ytableau}
  \caption{The Young diagram of $\lambda=(4,3,1)$ on the left and its conjugate
    $\lambda'=(3,2,2,1)$ on the right.}
  \label{fig:YD}
\end{figure}

For two partitions $\lambda$ and $\mu$, we write $\mu\subseteq\lambda$ if
$\mu_i\le \lambda_i$ for all $i\ge1$. In this case the \emph{skew shape} $\lm$
is defined to be the set-theoretic difference $\lambda-\mu$ of their Young
diagrams.

For a skew shape $\lm$ and an integer $k\ge1$, define $\row_{\le k}(\lm)$
(resp.~$\row_{\ge k}(\lm)$) to be the skew shape obtained from $\lm$ by taking
the rows of index $j\le k$ (resp.~$j\ge k$). Similarly, we define $\col_{\le
  k}(\lm)$ and $\col_{\ge k}(\lm)$ using columns. See Figure~\ref{fig:skew
  shape}.

\begin{figure}
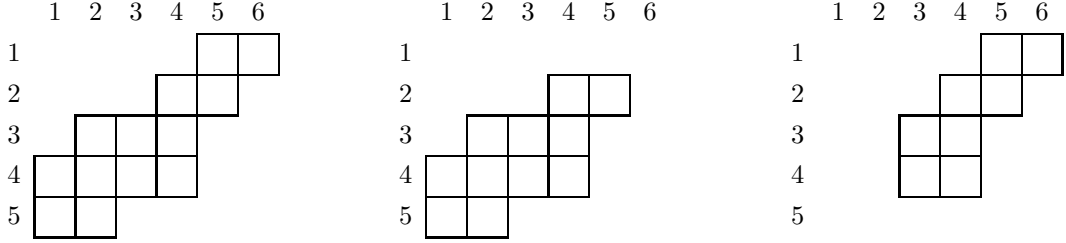

  \centering
  \begin{ytableau}
  \none& \none[1] & \none[2] & \none[3] &\none[4]  &\none[5]& \none[6]\\
  \none[1]& \none & \none & \none &\none  && \\
  \none[2]& \none&\none&\none & & \\
  \none[3]& \none & & & \\
  \none[4]& {} & & & \\
  \none[5]& {} & \\
  \end{ytableau}\qquad\qquad
  \begin{ytableau}
  \none& \none[1] & \none[2] & \none[3] &\none[4]  &\none[5]& \none[6]\\
  \none[1]&\none & \none & \none &\none  &\none& \none\\
  \none[2]&\none&\none&\none & & \\
  \none[3]&\none & & & \\
  \none[4]&{} & & & \\
  \none[5]&{} & \\
  \end{ytableau}\qquad\qquad
  \begin{ytableau}
  \none& \none[1] & \none[2] & \none[3] &\none[4]  &\none[5]& \none[6]\\
  \none[1]& \none & \none & \none &\none  && \\
  \none[2]&  \none&\none&\none & & \\
  \none[3]& \none &\none & & \\
  \none[4]& \none & \none& & \\
  \none[5]& \none & \none\\
  \end{ytableau}
  \caption{The skew diagram $\lm$ for $\lambda=(6,5,4,4,2)$ and $\mu=(4,3,1)$ on
    the left, $\row_{\ge2}(\lm)$ in the middle and $\col_{\ge3}(\lm)$ on the
    right. For visibility the row and column indices are written.}
  \label{fig:skew shape}
\end{figure}

Let $\rho$ be a finite subset of $\ZZ^+\times\ZZ^+$, where $\ZZ^+$ is the set of
positive integers. A \emph{tableau} of shape $\rho$ is just a map $T:\rho\to Z$,
where $Z$ is a linearly ordered set. If $T:\rho\to Z$ is a tableau we write
$\sh(T)=\rho$. 

A \emph{reverse plane partition (RPP)} of shape $\lm$ is a tableau
$R:\lm\to\ZZ^+$ such that the entries weakly increase in each row and
column, i.e., $R(i,j)\le R(i,j+1)$ and $R(i,j)\le R(i+1,j)$ whenever these
values are defined. The set of RPPs of shape $\lm$ is denoted by $\RPP(\lm)$.
For $R\in \RPP(\lm)$, the \emph{weight} of $R$ is defined by
\begin{equation}
  \label{eq:wtR}
\wt(R)=\prod_{i\ge1}x_i^{a_i(R)} t_i^{b_i(R)},
\end{equation}
where $a_i(R)$ is the number of columns containing an $i$ and $b_i(R)$ is the
number of cells $(i,j)$ such that $R(i,j)=R(i+1,j)$. For example if $R$ is the
RPP in Figure~\ref{fig:RPP}, then $\wt(R)=x_1^4x_2^3x_3^2t_1 t_3^2t_4$.

\begin{figure}
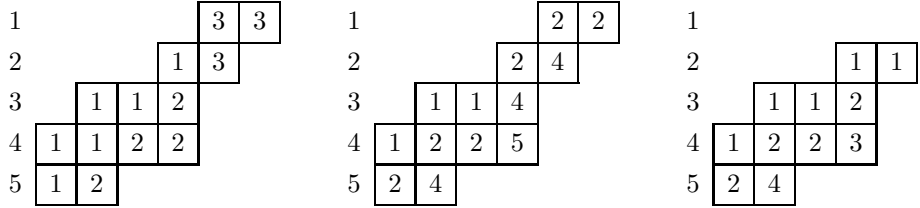

  \centering
  \begin{ytableau}
   \none[1]& \none & \none & \none &\none  &3&3 \\
   \none[2]& \none&\none&\none &1 &3 \\
   \none[3]& \none & 1& 1&2 \\
   \none[4]& 1 &1 &2 &2 \\
   \none[5]& 1 &2 \\
  \end{ytableau}\qquad
  \begin{ytableau}
   \none[1]&\none & \none & \none &\none  &2&2 \\
   \none[2]&\none&\none&\none &2 &4 \\
   \none[3]&\none & 1& 1& 4\\
   \none[4]&{1} &2 &2 &5 \\
   \none[5]&{2} &4 \\
  \end{ytableau}\qquad
  \begin{ytableau}
   \none[1]&\none & \none & \none &\none  &\none \\
   \none[2]&\none&\none&\none &1 &1 \\
   \none[3]&\none & 1& 1& 2\\
   \none[4]&{1} &2 &2 &3 \\
   \none[5]&{2} &4 \\
  \end{ytableau}
  \caption{An RPP of shape $(6,5,4,4,2)/(4,3,1)$ on the left, an SSYT of shape
    $(6,5,4,4,2)/(4,3,1)$ in the middle and an elegant tableau of shape $(6,5,4,4,2)/(6,4,3,1)$ on the right. The
    row indices are written on the left of each diagram.}
  \label{fig:RPP}
\end{figure}

A \emph{semistandard Young tableau (SSYT)} is an RPP with the extra condition
that the entries are strictly increasing in each column. The set of SSYTs of
shape $\lm$ is denoted by $\SSYT(\lm)$. An \emph{elegant tableau} is an SSYT $E$
of a certain skew shape $\lambda/\nu$ such that $1\le E(i,j)\le i-1$ for all
$(i,j)\in \lambda/\nu$. See Figure~\ref{fig:RPP}. Elegant tableaux were first
considered by Lenart \cite[Theorem~2.7]{Lenart2000} and further studied by Lam
and Pylyavskyy \cite{LP2007}.

Note that if $R\in\SSYT(\lm)\subseteq\RPP(\lm)$, then $R$ has no repeated
entries in each column and therefore the weight of $R$ defined in \eqref{eq:wtR}
is given by
\[
  \wt(R) = x_R=x_1^{c_1(T)}x_2^{c_2(T)}\cdots,
\]
where $c_i(T)$ is the number of $i$'s in $R$. For example, if $R$ is the SSYT in
Figure~\ref{fig:RPP}, then $\wt(R)=x_1^3x_2^6x_4^3x_5$.

Let $x=(x_1,x_2,\dots)$ and $t=(t_1,t_2,\dots)$ be sequences of variables. The
\emph{refined dual stable Grothendieck polynomial} $\gg_{\lm}(x;t)$ is defined
by
\[
  \gg_{\lambda/\mu}(x;t) = \sum_{R\in\RPP(\lm)} \wt(R).
\]

Now we recall the main result, Theorem~\ref{thm:main}: if $\lambda$ and $\mu$
are partitions with $\ell(\lambda')\le n$,
\[
  \widetilde{g}_{\lambda/\mu}(x;t) = \det
  \left( e_{\lambda'_i-\mu'_j-i+j}
    (x_1,x_2,\dots,t_{\mu'_j+1},t_{\mu'_j+2},\dots,t_{\lambda'_i-1})
  \right)_{1\le i,j\le n}.
\]
Note that if $\mu\not\subseteq\lambda$, by definition,
$\gg_{\lambda/\mu}(x;t)=0$. It is easy to see that in this case the above
determinant also vanishes because if $\mu\not\subseteq\lambda$, then
$\lambda'_r<\mu'_r$ for some $1\le r\le n$, which implies that the $(i,j)$ entry
is zero for all $r+1\le i\le n$ and $1\le j\le r$. Moreover, it is also easy to
see that if $\ell(\lambda')=m<n$, then the above determinant is equal to its
principal minor consisting of the first $m$ rows and columns. Therefore it is
sufficient to show Theorem~\ref{thm:main} for the case $\mu\subseteq\lambda$ and
$\ell(\lambda')=n$. From now on we always assume that $\mu\subseteq\lambda$ and
$\ell(\lambda')=n$.

Let $\omega$ be the smallest infinite ordinal number and let
\begin{align*}
  \NN &= \{0,1,2,\dots\},\\
  \NN_\omega &= \{0,1,2,\dots,\omega,\omega+1,\omega+2,\dots,2\omega\},\\
  G &= \NN\times\NN_\omega,
\end{align*}
where the numbers are ordered as usual by
\[
0<1<2<\cdots<\omega<\omega+1<\omega+2<\cdots<2\omega.
\]
We will also write $\omega+i$ as $i^*$.

A \emph{path} from $(a,0)$ to $(b,2\omega)$ is a pair $(s_1,s_2)$ of infinite
sequences $s_j=((u_{0}^{(j)},v_{0}^{(j)}), (u_{1}^{(j)},v_{1}^{(j)}),\dots)$,
$j=1,2$, of points in $G$ satisfying the following conditions:
\begin{itemize}
\item The steps $(u_{i+1}^{(j)},v_{i+1}^{(j)})- (u_{i}^{(j)},v_{i}^{(j)})$, for
  $i\ge0$ and $j=1,2$, consist of \emph{up steps} $(0,1)$ and \emph{diagonal
    steps} $(1,1)$.
\item $(u_0^{(1)},v_0^{(1)})=(a,0)$, $v_0^{(2)}=\omega$ and
  \[
    \lim_{m\to\infty} u_m^{(1)} = u_0^{(2)}, \qquad
    \lim_{m\to\infty} u_m^{(2)} = b.
  \]
\end{itemize}
The \emph{weight} of a path $p$ is defined to be
\[
\wt(p) = \prod_{i\ge 1}x_i^{a_i(p)} t_i^{b_i(t)},
\]
where $a_i(p)$ (resp.~$b_i(p)$) is the set of diagonal steps of $p$ ending at
height $i$ (resp.~$\omega+i$). See Figure~\ref{fig:L(i,j)}.

Suppose that $\lambda$ and $\mu$ are partitions with $\mu\subseteq\lambda$ and
$\ell(\lambda')\le n$. Denote by $\LL(i,j)$ the set of all paths from
$(\mu'_i+n-i,0)$ to $(\lambda'_j+n-j,2\omega)$ in which there is no diagonal
step between the lines $y=\omega$ and $y=\omega+\mu'_i$ and no diagonal step
above the line $y=\omega+\lambda'_i-1$. See Figure~\ref{fig:L(i,j)} for a
typical example of a path in $\LL(i,j)$. It is clear from the construction that
\begin{equation}
  \label{eq:e}
  \sum_{p\in \LL(i,j)} \wt(p)
  =e_{\lambda'_j-\mu'_i-j+i}(x_1,x_2,\dots,t_{\mu'_i+1},t_{\mu'_i+2},\dots, t_{\lambda'_j-1}).
\end{equation}

\begin{figure}
  \centering
\begin{tikzpicture}[x=0.4cm, y=0.4cm]
\RSEgrid{9}66
\node[below] at (2,0){$(2,0)$};
\node[above] at (7,16){$(7,2\omega)$};
\filldraw (2,0) circle [radius=2pt];
\filldraw (7,16) circle [radius=2pt];

  \draw[line width=2pt] (2,0)--
  ++(0,1)--
  node[above]{$2$} ++(1,1)--
  ++(0,2)--
  node[above]{$5$} ++(1,1)--
  ++(0,3)--
  node[above]{$1^*$} ++(1,1)--
  ++(0,2)--
  node[above]{$4^*$} ++(1,1)--
  node[above]{$5^*$} ++(1,1)--
  ++(0,3); 
\end{tikzpicture}\qquad\qquad
\begin{tikzpicture}[x=0.4cm, y=0.4cm]
\filldraw[lightgray] (0,8) rectangle (9,11);
\filldraw[lightgray] (0,13) rectangle (9,16);
\RSEgridonly{9}66
\node[left] at (0,0) {$0$};
\node[left] at (0,8) {$\omega$};
\node[left] at (0,7.3) {$\vdots$};
\node[left] at (0,16) {$2\omega$};
\node[left] at (0,15.3) {$\vdots$};
\node[left] at (0,11) {$\omega+\mu'_i$};
\node[left] at (0,13) {$\omega+\lambda'_i-1$};
\node[below] at (4,0){$(\mu_i'+n-i,0)$};
\node[above] at (7,16){$(\lambda_j'+n-j,2\omega)$};
\filldraw (4,0) circle [radius=2pt];
\filldraw (7,16) circle [radius=2pt];

  \draw[line width=2pt] (4,0)--
  ++(0,1)--
  ++(1,1)--
  ++(0,2)--
  ++(1,1)--
  ++(0,6)--
  ++(1,1)--
  ++(0,4);
\end{tikzpicture}
\caption{The left diagram is a path from $(2,0)$ to $(7,2\omega)$ with weight
  $x_2x_5t_1t_4t_5$. The height of the ending point of each diagonal step is
  shown. The right diagram illustrates a typical path in $\LL(i,j)$, which
  cannot have diagonal steps in the gray areas.}
  \label{fig:L(i,j)}
\end{figure}
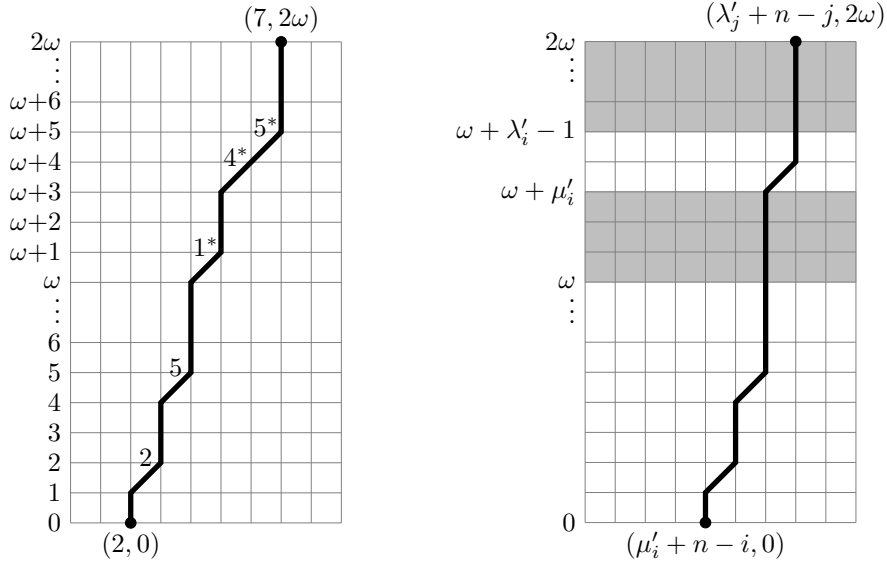

An \emph{$n$-path} is an $n$-tuple of paths. Denote by $\Sym_n$ the set of
permutations on $\{1,2,\dots,n\}$. For a permutation $\pi\in\mathfrak{S}_n$, let
$\LL(\pi)$ denote the set of $n$-paths $\pp= (p_1,\dots,p_n)$ such that $p_i\in
\LL(i,\pi(i))$ for all $1\le i\le n$. Note that $\LL(\pi)$ may be
empty for some $\pi$. Define
\[
\LL=\bigcup_{\pi\in \mathfrak{S}_n}\LL(\pi).
\]

The \emph{type} of an $n$-path $\pp$ in $\LL$, denoted $\type(\pp)$, is the
permutation $\pi$ for which $\pp\in\LL(\pi)$. Note that $\type(\pp)$ is uniquely
determined because the starting points $(\mu'_i+n-i,0)$ and the ending points
$(\lambda'_j+n-j,2\omega)$ are all distinct. The \emph{weight} of $\pp=
(p_1,\dots,p_n)\in \LL$ is defined by
\[
  \wt(\pp)=\sign(\type(\pp))\wt(p_1)\cdots\wt(p_n).
\]

The following notation will be used throughout this paper. See
Figure~\ref{fig:notation} for an illustration.

\begin{notation}\label{notation}
  Fix partitions $\lambda$ and $\mu$ with $\mu\subseteq\lambda$,
  $\ell(\lambda)=\ell$, and $\ell(\lambda')=\lambda_1=n$. Define
  $d_1>d_2>\dots>d_r$ to be the distinct integers in
  $\{\mu_1,\mu_2,\dots,\mu_{\ell(\mu)}\}$ and let $d_0=n$ and $d_{r+1}=0$. For
  $1\le i\le r+1$, 
  \begin{itemize}
  \item $m_i$ denotes the multiplicity of the part $d_i$ in $\mu$, where the
    multiplicity of $d_{r+1}=0$ in $\mu$ is defined to be
    $\ell(\lambda)-\ell(\mu)$,
  \item $M_i=m_1+\dots+m_i$, and 
  \item $D_i=\{d_{i}+1,d_{i}+2,\dots,d_{i-1}\}$.
  \end{itemize}
\end{notation}

\begin{figure}
  \centering
  \begin{tikzpicture}[scale=0.5](0,0)(10,11)
\draw[line width=1pt](0,0) --++(6,0) --++(0,1) --++(2,0)--++(0,1)--++(1,0)--++(0,2)
--++(2,0)--++(0,2)--++(1,0)--++(0,1)--++(2,0)--++(0,2)--++(1,0)--++(0,1)--++(-15,0)--++(0,-10);

\draw[line width=1pt](0,2)--++(3,0)--++(0,2)--++(2,0)--++(0,.5);
\draw[line width=1pt](6,5.5)--++(0,0.5)--++(2,0)--++(0,2)--++(2,0)--++(0,2);
\filldraw (5.3,4.8) circle [radius=1pt];
\filldraw (5.5,5) circle [radius=1pt];
\filldraw (5.7,5.2) circle [radius=1pt];

\draw[dotted] (0,4)--(3,4);
\draw[dotted] (0,6)--(7,6);
\draw[dotted] (0,8)--(8,8);
\draw[dotted] (3,4)--(3,10);
\draw[dotted] (6,6)--(6,10);
\draw[dotted] (8,8)--(8,10);

\draw [decorate,decoration={brace,amplitude=5pt},xshift=3pt,yshift=0pt]
(10,10) -- (10,8) node [midway,xshift=.5cm] {$m_1$}; 
\draw [decorate,decoration={brace,amplitude=5pt},xshift=3pt,yshift=0pt]
(8,8) -- (8,6) node [midway,xshift=.5cm] {$m_2$}; 
\draw [decorate,decoration={brace,amplitude=5pt},xshift=3pt,yshift=0pt]
(3,4) -- (3,2) node [midway,xshift=.5cm] {$m_r$}; 
\draw [decorate,decoration={brace,amplitude=5pt},xshift=3pt,yshift=0pt]
(0,2) -- (0,0) node [midway,xshift=.7cm] {$m_{r+1}$}; 

\draw [decorate,decoration={brace,amplitude=5pt},xshift=-4pt,yshift=0pt]
(0,8) -- (0,10) node [midway,xshift=-0.5cm] {$M_1$}; 
\draw [decorate,decoration={brace,amplitude=5pt},xshift=-4pt,yshift=0pt]
(-1.5,6) -- (-1.5,10) node [midway,xshift=-0.5cm] {$M_2$}; 
\draw [decorate,decoration={brace,amplitude=5pt},xshift=-4pt,yshift=0pt]
(-3,2) -- (-3,10) node [midway,xshift=-0.5cm] {$M_r$}; 
\draw [decorate,decoration={brace,amplitude=5pt},xshift=-4pt,yshift=0pt]
(-4.5,0) -- (-4.5,10) node [midway,xshift=-1cm] {$M_{r+1}=\ell$}; 

\draw [decorate,decoration={brace,amplitude=5pt},yshift=1pt,yshift=0pt]
(0,8) -- (10,8) node [midway,yshift=0.4cm] {$d_{1}$}; 
\draw [decorate,decoration={brace,amplitude=5pt},yshift=1pt,yshift=0pt]
(0,6) -- (8,6) node [midway,yshift=0.4cm] {$d_{2}$}; 
\draw [decorate,decoration={brace,amplitude=5pt},yshift=1pt,yshift=0pt]
(0,2) -- (3,2) node [midway,yshift=0.4cm] {$d_r$}; 

\draw [decorate,decoration={brace,amplitude=5pt},yshift=3pt]
(0,10) -- (3,10) node [midway,yshift=0.4cm] {$D_{r}$}; 
\draw [decorate,decoration={brace,amplitude=5pt},yshift=3pt]
(8,10) -- (10,10) node [midway,yshift=0.4cm] {$D_2$}; 
\draw [decorate,decoration={brace,amplitude=5pt},yshift=3pt]
(10,10) -- (15,10) node [midway,yshift=0.4cm] {$D_1$}; 
\draw [decorate,decoration={brace,amplitude=5pt},yshift=3pt]
(0,11.5) -- (15,11.5) node [midway,yshift=0.4cm] {$d_0=n$}; 
  \end{tikzpicture}
  \caption{An illustration of Notation~\ref{notation} for a given $\lm$. Every
    letter is an integer except $D_i$'s, which are sets of column indices.}
  \label{fig:notation}
\end{figure}
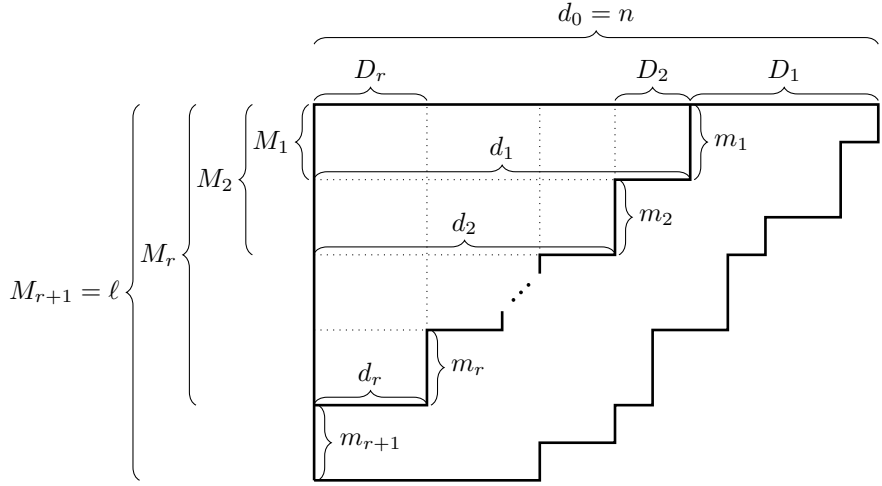

Note that for $1\le i,j\le n$, we have $\mu'_i=\mu'_{j}$ if and only if $i,j\in
D_k$ for some $1\le k\le r$.

Suppose $\pp= (p_1,\dots,p_n)\in \LL$. We say that $\pp$ is \emph{noncrossing}
if $p_i$ and $p_j$ have no common points for all $i\ne j$, and that $\pp$ is
\emph{semi-noncrossing} if $p_i$ and $p_j$ have no common points whenever $i$
and $j$ are distinct elements in $D_k$ for some $1\le k\le r$. Denote by $\lnc$
(resp.~$\lsnc$) the set of noncrossing (resp.~semi-noncrossing) paths in $\LL$.
Note that if $\mu=\emptyset$, then $\lnc=\lsnc$.

By expanding the determinant in Theorem~\ref{thm:main} (more precisely, the
determinant of the transpose of the matrix) using \eqref{eq:e}, we have
\begin{equation}
  \label{eq:1}
\det \left(
    e_{\lambda'_i-\mu'_j-i+j}(x_1,x_2,\dots,t_{\mu'_j+1},t_{\mu'_j+2},\dots,
    t_{\lambda'_i-1}) \right)_{1\le i,j\le n}  =\sum_{\pp\in \LL} \wt(\pp). 
\end{equation}
The standard method of the Lindstr\"om--Gessel--Viennot lemma
\cite{LGV,Lindstrom} interprets a determinant as a weighted sum of noncrossing
$n$-paths via a sign-reversing involution which exchanges ``tails'' of
intersecting paths. Roughly speaking, in order for this to work ``local''
changes of the steps in an $n$-path must be allowed. Such ``local'' changes are
not allowed for an $n$-path $\pp=(p_1,\dots,p_n)$ in $\LL$ because each path
$p_i$ has the ``global'' restriction that there are no diagonal steps between
the lines $y=\omega$ and $y=\omega+\mu'_i$. For example, see
Figure~\ref{fig:forbidden}. 

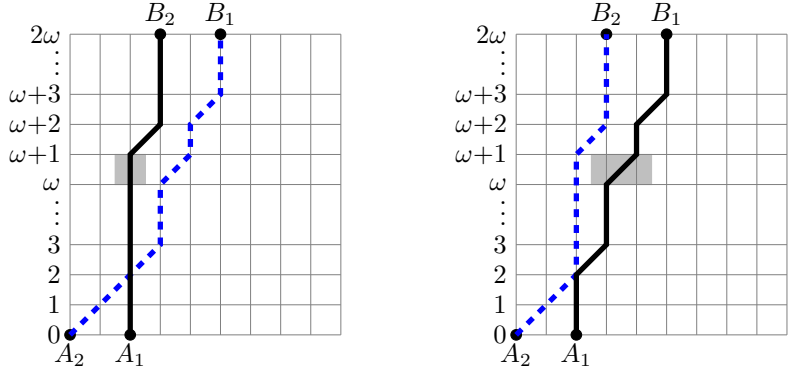
\begin{figure}
  \centering
\begin{tikzpicture}[x=0.4cm, y=0.4cm]
\filldraw[lightgray] (1.5,5) rectangle (2.5,6);
\RSEgrid{9}33
\node[below] at (0,0){$A_2$};
\node[below] at (2,0){$A_1$};
\node[above] at (5,10){$B_1$};
\node[above] at (3,10){$B_2$};
\filldraw (0,0) circle [radius=2pt];
\filldraw (2,0) circle [radius=2pt];
\filldraw (5,10) circle [radius=2pt];
\filldraw (3,10) circle [radius=2pt];

\draw[line width=2pt,blue,dashed] (0,0)-- ++(1,1)-- ++(1,1)-- ++(1,1)--
++(0,2)-- ++(1,1)-- ++(0,1)-- ++(1,1)-- ++(0,2);
\draw[line width=2pt] (2,0)-- ++(0,2)-- ++(0,4)-- ++(1,1)-- ++(0,3);
\end{tikzpicture}\qquad\qquad
\begin{tikzpicture}[x=0.4cm, y=0.4cm]
\filldraw[lightgray] (2.5,5) rectangle (4.5,6);
\RSEgrid{9}33
\node[below] at (0,0){$A_2$};
\node[below] at (2,0){$A_1$};
\node[above] at (5,10){$B_1$};
\node[above] at (3,10){$B_2$};
\filldraw (0,0) circle [radius=2pt];
\filldraw (2,0) circle [radius=2pt];
\filldraw (5,10) circle [radius=2pt];
\filldraw (3,10) circle [radius=2pt];

\draw[line width=2pt,blue,dashed] (0,0)-- ++(1,1)-- ++(1,1)-- 
++(0,4)-- ++(1,1)-- ++(0,3);
\draw[line width=2pt] (2,0)-- ++(0,2)--
++(1,1)--++(0,2)-- ++(1,1)-- ++(0,1)-- ++(1,1)-- ++(0,2);
\end{tikzpicture}
\caption{Let $\lambda=(2,2,2,1)$ and $\mu=(1)$ so that $\lambda'_1=4$,
  $\lambda'_2=3$ and $\mu'_1=1$. The left diagram shows a $2$-path
  $(p_1,p_2)\in\LL$, where $p_i$ is the path from $A_i$ to $B_i$ for $i=1,2$. If
  we switch the tails of $p_1$ and $p_2$ after their unique intersection, we
  obtain the resulting $2$-path $(p'_1,p'_2)$ as shown on the right. The gray
  area shows (part of) the restriction on the path whose starting point is
  $A_1=(\mu_1+2-1,0)=(2,0)$. Since $p'_1$ has a diagonal step in the gray area,
  $(p'_1,p'_2)\not\in \LL$.}
  \label{fig:forbidden}
\end{figure}

However, it is possible to cancel
all $n$-paths except for the semi-noncrossing $n$-paths.

\begin{prop}\label{prop:snc}
  Let $\lambda$ and $\mu$ be partitions with $\mu\subseteq\lambda$,
  $\ell(\lambda)=\ell$, and $\ell(\lambda')=n$. Then
\[
  \det \left(
    e_{\lambda'_i-\mu'_j-i+j}(x_1,x_2,\dots,t_{\mu'_j+1},t_{\mu'_j+2},\dots,
    t_{\lambda'_i-1}) \right)_{1\le i,j\le n} = \sum_{\pp\in \lsnc} \wt(\pp).
\]
\end{prop}
\begin{proof}
  By \eqref{eq:1} it is sufficient to show that
\begin{equation}
  \label{eq:13}
   \sum_{\pp\in \LL} \wt(\pp) = \sum_{\pp\in \lsnc} \wt(\pp).  
\end{equation}
We will cancel all paths in $\LL\setminus \lsnc$ using the standard method of
switching tails of two paths. More precisely, suppose
$\pp=(p_1,\dots,p_n)\in\LL\setminus \lsnc$. Then we can find the smallest
integer $k$ such that $p_i$ and $p_j$ have common points for some $i\ne j$ in
$D_k$. Choose such $i$ and $j$ so that $(i,j)$ is the smallest in the
lexicographic order. Let $(a,b)$ be the last intersection of $p_i$ and $p_j$.
Let $p'_i$ and $p'_j$ be the paths obtained from $p_i$ and $p_j$ respectively by
exchanging the subpaths after $(a,b)$. If $\type(\pp)=\pi$, then $p_i\in
\LL(i,\pi(i))$ and $p_j\in \LL(j,\pi(j))$. Since $i,j\in D_k$, we have
$\mu'_i=\mu'_j$. Therefore neither $p_i$ nor $p_j$ has diagonal steps between
heights $\omega$ and $\omega+\mu'_i=\omega+\mu'_j$, which ensures that $p'_i\in
\LL(i,\pi(j))$ and $p'_j\in \LL(j,\pi(i))$. Let $\pp'$ be the $n$-path obtained
from $\pp$ by replacing $p_i$ and $p_j$ by $p'_i$ and $p'_j$ respectively. Then
$\pp\in \LL\setminus\lsnc$ and $\type(\pp')=\pi(i,j)$, where $(i,j)$ is the
transposition. Therefore $\wt(\pp)=-\wt(\pp')$. It is easily seen that this
argument shows that $\sum_{\pp\in \LL\setminus\lsnc} \wt(\pp)=0$, hence
\eqref{eq:13}.
\end{proof}

\section{Vertical tableaux and $n$-paths}
\label{sec:vertical-tableaux}

In this section we introduce a notion of vertical tableaux and give a simple
bijection between them and certain $n$-paths.

A \emph{composition} is a sequence $\alpha=(\alpha_1,\alpha_2,\dots,\alpha_n)$
of nonnegative integers. The \emph{vertical diagram} of a composition $\alpha$
is defined by
\[
V(\alpha) = \{(i,j)\in\ZZ^2: 1\le j\le n, 1\le i\le \alpha_j\}.
\]
Similarly to Young diagrams each element $(i,j)$ in the vertical diagram is
represented by a cell in row $i$ and column $j$. Since $\lambda=V(\lambda')$ as
subsets of $\ZZ^2$, we will also consider the Young diagram of $\lambda$ as a
vertical diagram. The notation used for Young diagrams is naturally extended to
vertical diagrams. For example, for a vertical diagram $V$, define $\col_{\ge
  k}(V)=\{(i,j)\in V: j\ge k\}$, and for two vertical diagrams $V_1$ and $V_2$
with $V_1\subseteq V_2$, define $V_2/V_1$ to be the set-theoretic difference
$V_2-V_1$. We say that $V_1$ and $V_2$ are the \emph{inner shape} and the
\emph{outer shape} of $V_2/V_1$, respectively. See Figures~\ref{fig:VD} and
\ref{fig:VD2}.

\begin{figure}
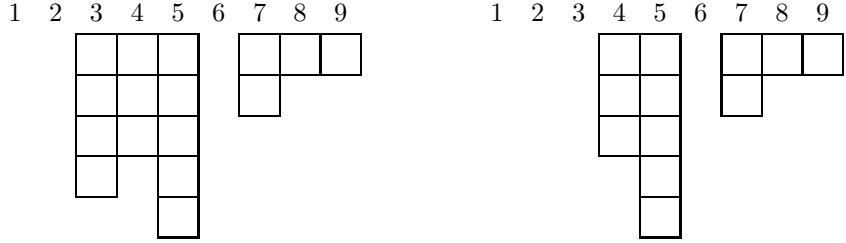

  \centering
   \begin{ytableau}
     \none[1] & \none[2] & \none[3] & \none[4] & \none[5] & \none[6]& \none[7]& \none[8]& \none[9]\\
     \none&\none&&&&\none&&&\\
     \none&\none&&&&\none&&\none&\none\\
     \none&\none&&&&\none&\none&\none&\none\\
     \none&\none&&\none&&\none&\none&\none&\none\\
     \none&\none&\none&\none&&\none&\none&\none&\none\\
   \end{ytableau} \qquad\qquad
   \begin{ytableau}
     \none[1] & \none[2] & \none[3] & \none[4] & \none[5] & \none[6]& \none[7]& \none[8]& \none[9]\\
     \none&\none&\none&&&\none&&&\\
     \none&\none&\none&&&\none&&\none&\none\\
     \none&\none&\none&&&\none&\none&\none&\none\\
     \none&\none&\none&\none&&\none&\none&\none&\none\\
     \none&\none&\none&\none&&\none&\none&\none&\none\\
   \end{ytableau} 
   \caption{The vertical diagram $V(\alpha)$ on the left and the vertical
     diagram $\col_{\ge4}(V(\alpha))$ on the right for the composition
     $\alpha=(0,0,4,3,5,0,2,1,1)$. For visibility the column indices are written
     above the diagrams.}
  \label{fig:VD}
\end{figure}

\begin{figure}
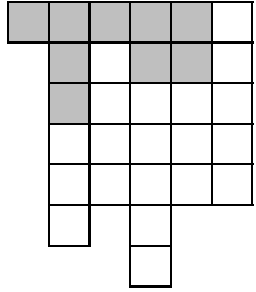

  \centering
   \begin{ytableau}
      \muentry{} &  \muentry{} & \muentry{} & \muentry{} & \muentry{} & \\
      \none &  \muentry{} & & \muentry{} & \muentry{} & \\
      \none &  \muentry{} & & & & \\
      \none & & & & & \\
      \none & & & & & \\
      \none & & \none & & \none & \none\\
      \none & \none & \none & & \none & \none\\
   \end{ytableau} 
   \caption{The diagram $V(\beta)/V(\alpha)$ for $\alpha=(1,3,1,2,2,0)$ and
     $\beta=(1,6,5,7,5,5)$ is shown with the white cells.}
  \label{fig:VD2}
\end{figure}

 \begin{figure}
   \centering
\begin{tikzpicture}[x=0.4cm, y=0.4cm]
\RSEgrid{12}54
\node[below] at (0,0){$A_6$};
\node[below] at (3,0){$A_5$};
\node[below] at (4,0){$A_4$};
\node[below] at (4,-1){$A_3$};
\node[below] at (7,0){$A_2$};
\node[below] at (6,0){$A_1$};
\node[above] at (5,13){$B_6$};
\node[above] at (6,13){$B_5$};
\node[above] at (6,14){$B_1$};
\node[above] at (8,13){$B_3$};
\node[above] at (9,13){$B_4$};
\node[above] at (10,13){$B_2$};
\filldraw (0,0) circle [radius=2pt];
\filldraw (3,0) circle [radius=2pt];
\filldraw (4,0) circle [radius=2pt];
\filldraw (6,0) circle [radius=2pt];
\filldraw (7,0) circle [radius=2pt];
\filldraw (5,13) circle [radius=2pt];
\filldraw (6,13) circle [radius=2pt];
\filldraw (8,13) circle [radius=2pt];
\filldraw (9,13) circle [radius=2pt];
\filldraw (10,13) circle [radius=2pt];

  \draw[line width=2pt] (0,0)--
  ++(0,2)--
  node[above]{$3$} ++(1,1)--
  node[above]{$4$} ++(1,1)--
  ++(0,3)--
  node[above]{$1^*$} ++(1,1)--
  node[above]{$2^*$} ++(1,1)--
  ++(0,1)--
  node[above]{$4^*$} ++(1,1)--
  ++(0,2);
  
  \draw[line width=2pt] (3,0)--
  ++(0,1)--
  node[above]{$2$} ++(1,1)--
  ++(0,2)--
  node[above]{$4$} ++(1,1)--
  ++(0,2)--
  node[above]{$1^*$} ++(1,1)--
  ++(0,5);
  
  \draw[line width=2pt] (4,0)--
  node[above]{$1$} ++(1,1)--
  ++(0,1)--
  node[above]{$3$} ++(1,1)--
  node[above]{$4$} ++(1,1)--
  ++(0,4)--
  node[above]{$2^*$} ++(1,1)--
  node[above]{$3^*$} ++(1,1)--
  ++(0,3);

  \draw[line width=2pt,blue,dashed] (4,0)--
  node[below]{$1$} ++(1,1)--
  node[below]{$2$} ++(1,1)--
  ++(0,1)--
  node[below]{$4$} ++(1,1)--
  node[below]{$5$} ++(1,1)--
  ++(0,8);

  \draw[line width=2pt,red,dashed] (6,0)--
  ++(0,13);

  \draw[line width=2pt,gray,dashed] (7,0)--
  node[below]{$1$} ++(1,1)--
  ++(0,8)--
  node[below]{$3^*$} ++(1,1)--
  node[below]{$4^*$} ++(1,1)--
  ++(0,2);

\end{tikzpicture}\qquad\qquad
   \begin{ytableau}
      \muentry{} &  \muentry{} & \muentry{} & \muentry{} & \muentry{} &3 \\
      \none &  \muentry{} & 1& \muentry{} & \muentry{} & 4\\
      \none &  \muentry{} &2 &1 &2 &1^* \\
      \none &1 &4 &3 &4 &2^* \\
      \none &3^* &5 &4 &1^* &4^* \\
      \none &4^* & \none & 2^*& \none & \none\\
      \none & \none & \none &3^* & \none & \none\\
      \none\\
      \none\\
   \end{ytableau} 
   \caption{On the left is a $6$-path $\pp=(p_1,\dots,p_6)\in\L(\alpha,\beta)$
     for $\alpha=(1,3,1,2,2,0)$ and $\beta=(1,6,5,7,5,5)$. Each $p_i$ is a path
     from $A_i=(\alpha_i+6-i,0)$ to $B_i=(\beta_i+6-i,2\omega)$. Its corresponding
     vertical tableau $\Tab(\pp)$ is shown on the right.}
   \label{fig:VT}
 \end{figure}
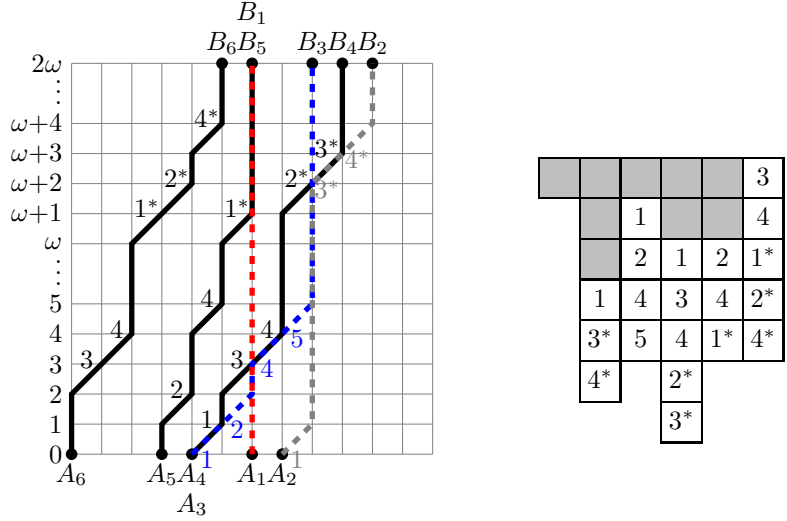

 For vertical diagrams $V_1$ and $V_2$ with $V_1\subseteq V_2$, a \emph{vertical
   tableau} of shape $V_2/V_1$ is a filling of $V_2/V_1$ with numbers in
 $\{1<2<\cdots<1^*<2^*<\cdots\}$ such that the entries are strictly increasing
 in each column. See the right diagram in Figure~\ref{fig:VT} for an example of
 a vertical tableau. Let $\VT(V_2/V_1)$ denote the set of vertical tableaux of
 shape $V_2/V_1$.

\begin{defn}\label{def:Tab}[The map $\Tab$ sending $n$-paths to vertical
  tableaux]
  Let $\alpha=(\alpha_1,\dots,\alpha_n)$ and $\beta=(\beta_1,\dots,\beta_n)$ be
  compositions with $V(\alpha)\subseteq V(\beta)$. Define $\L(\alpha,\beta)$ to
  be the set of $n$-paths $\pp=(p_1,\dots,p_n)$, where $p_i$ is a path from
  $(\alpha_i+n-i,0)$ to $(\beta_i+n-i,2\omega)$.

  For $\pp=(p_1,\dots,p_n)\in\L(\alpha,\beta)$, define $\Tab(\pp)$ to be the
  vertical tableau $T\in \VT(V(\beta)/V(\alpha))$ constructed as follows. For
  each diagonal step of $p_i$, if its ending point is $(a,b)$, fill the
  $(a-n+i-1,i)$-entry of $T$ with $b$.
\end{defn}

See Figure~\ref{fig:VT} for an example of the map $\Tab$ in
Definition~\ref{def:Tab}. The following proposition is straightforward to
verify.

\begin{prop}\label{prop:Tab}
  Following the notation in Definition \ref{def:Tab}, the map $\Tab$ is a
  bijection from $\L(\alpha,\beta)$ to $\VT(V(\beta)/V(\alpha))$. Moreover,
  if $\Tab(\pp)=T$, then for every positive integer $h$ the total number of
  diagonal steps in $\pp$ ending at height $h$ (resp.~$\omega+h$) is equal to
  the number of times $h$ (resp.~$h^*$) appears in $T$.
\end{prop}

For a partition $\lambda$ with $\ell(\lambda')=n$ and a permutation
$\pi\in\Sym_n$, we define $\pi(\lambda)$ to be the vertical diagram given by 
\[
\pi(\lambda) = \{(i,j)\in\ZZ^2: 1\le j\le n, 1\le i\le \lambda'_{\pi_j}-\pi_j+j\}.
\]
Note that if $\pi$ is the identity permutation then $\pi(\lambda)$ is the Young
diagram of $\lambda$. One may worry about the situation that
$\lambda'_{\pi_j}-\pi_j+j<0$ in the definition of $\pi(\lambda)$. Since we will
only consider $\pi(\lambda)$ when $\VT(\pi(\lambda)/\mu)$ is nonempty (or
equivalently, when $\mu\subseteq\pi(\lambda)$) this will never occur, see the
paragraph after the proof of Lemma~\ref{lem:VT}.

The following lemma shows that the type of $\pp\in\LL$ is encoded in the outer
shape of the vertical tableau $\Tab(\pp)$ while the inner shape of $\Tab(\pp)$
is always $\mu$. See Figure~\ref{fig:ppT} for an example.

\begin{lem}\label{lem:VT}
  For $\pp\in\LL$, we have $\type(\pp)=\pi$ if and only if
  $\Tab(\pp)\in\VT(\pi(\lambda)/\mu)$.
\end{lem}
\begin{proof}
  Suppose that $\pp=(p_1,\dots,p_n)\in\LL$ has $\type(\pp)=\pi$. Let $\alpha$
  and $\beta$ be the compositions given by $\alpha_i=\mu'_i$ and
  $\beta_i=\lambda'_{\pi_i}-\pi_i+i$. Then
  $V(\beta)/V(\alpha)=\pi(\lambda)/\mu$. Since $p_i$ is a path from
  $(\mu'_i+n-i,0)=(\alpha_i+n-i,0)$ to
  $(\lambda'_{\pi_i}+n-\pi_i,2\omega)=(\beta_i+n-i,2\omega)$, we have
  $\pp\in\Tab(\beta/\alpha)$. By Proposition~\ref{prop:Tab},
  $\Tab(\pp)\in\VT(V(\beta)/V(\alpha))=\VT(\pi(\lambda)/\mu)$.

  Conversely, suppose that $\pp\in\LL$ satisfies
  $\Tab(\pp)\in\VT(\pi(\lambda)/\mu)$. Let $\type(\pp)=\sigma$. Then by what we
  just proved, we obtain $\Tab(\pp)\in\VT(\sigma(\lambda)/\mu)$, which implies
  $\sigma(\lambda)=\pi(\lambda)$, or equivalently,
  \[
    (\lambda'_{\pi_1}-\pi_1+1,\dots,\lambda'_{\pi_n}-\pi_n+n)
    =(\lambda'_{\sigma_1}-\sigma_1+1,\dots, \lambda'_{\sigma_n}-\sigma_n+n). 
\]
By subtracting $i$ from the $i$th component we also have
  \[
    (\lambda'_{\pi_1}-\pi_1,\dots,\lambda'_{\pi_n}-\pi_n)
    =(\lambda'_{\sigma_1}-\sigma_1,\dots, \lambda'_{\sigma_n}-\sigma_n).
\]
Both sequences in the above equation are rearrangements of
$(\lambda'_1-1,\dots,\lambda'_n-n)$, which is a strictly decreasing sequence.
Since there are no repeated entries in this sequence, the rearrangements must be
identical and we obtain $\pi=\sigma$. Hence $\type(\pp)=\pi$ and the proof is
completed.
\end{proof}

Note that in the proof of the above lemma if $\type(\pp)=\pi$, we must have
$\lambda'_{\pi_i}+n-\pi_i\ge \mu'_i+n-i$, or, equivalently, $\lambda'_{\pi_i}-\pi_i+i\ge
\mu'_i\ge0$. Hence in this case we always have $\lambda'_{\pi_j}-\pi_j+j\ge0$ in
the definition of $\pi(\lambda)$.

\section{RSE-tableaux and bijections of Lam and Pylyavskyy}
\label{sec:rse-tableaux}

In this section we define RSE-tableaux and two maps $\pu$ and $\pd$ on these
objects. The notion of RSE-tableaux was introduced implicitly by Lam and
Pylyavskyy \cite[Proof of Theorem~9.8]{LP2007} in their bijection between RPPs
and pairs of SSYTs and elegant tableaux. The maps $\pu$ and $\pd$ described in
this section are intermediate steps in their bijection. We assume reader's
familiarity with the RSK algorithm and its basic properties. See
\cite[Section~1.1]{Fulton1997} or \cite[Section~7.11]{EC2} for a standard
reference. In particular, we will use Row Bumping Lemma and Proposition in
\cite[Section~1.1]{Fulton1997}.

\begin{defn}
An \emph{RSE-tableau of shape $\lambda$ of level $k$} is a pair $T=(R,E)$ satisfying
the following conditions:
\begin{itemize}
\item $R\in\RPP(\nu)$ with row $k$ of $R$ marked, for some partition
  $\nu\subseteq\lambda$ satisfying $\row_{\le k}(\nu)=\row_{\le k}(\lambda)$,
\item $\row_{\ge k}(R)$ is an SSYT, and
\item $E$ is an elegant tableau of shape $\lambda/\nu$
  such that every entry in $E$ is at least $k$.
\end{itemize}
The set of RSE-tableaux of shape $\lambda$ and level $k$ is denoted by
$\RSE_k(\lambda)$.
\end{defn}

We will represent an RSE-tableau $T=(R,E)$ as the tableau obtained by drawing
both $R$ and $E$ in which every entry $i$ in $E$ is written as $i^*$. See
Figure~\ref{fig:RSE} for an example of an RSE-tableau. 

By definition, if $T=(R,E)$ is an RSE-tableau of level $k$, then $\row_{\le
  k}(R)$, $\row_{\ge k}(R)$, and $E$ are, respectively, an RPP, an SSYT, and an
elegant tableau. The name ``RSE'' stands for the initials of these three
tableaux.

\begin{figure}
  \centering
  \begin{ytableau}
    1&1&1&2&3\\
    1&1&2&3&3\\
    1&2&2&3&4&\none[\star]\\
    2&3\\
    3&4\\
    \none\\
  \end{ytableau}\qquad\qquad
  \begin{ytableau}
    \muentry{}&\muentry{}&\muentry{}&\muentry{}&\muentry{}\\
    \muentry{}&\muentry{}&\muentry{}&\muentry{}&\muentry{}\\
    \muentry{}&\muentry{}&\muentry{}&\muentry{}&\muentry{}\\
    \muentry{}&\muentry{}&3^*&3^*\\
    \muentry{}&\muentry{}&4^*&4^*\\
    3^*&5^*\\
  \end{ytableau}\qquad\qquad
  \begin{ytableau}
    1&1&1&2&3\\
    1&1&2&3&3\\
    1&2&2&3&4&\none[\star]\\
    2&3&3^*&3^*\\
    3&4&4^*&4^*\\
    3^*&5^*\\
  \end{ytableau}
  \caption{An RPP $R$ of shape $\nu$ on the left with row $3$ marked, an elegant
    tableau $E$ of shape $\lambda/\nu$ in the middle, and the RSE-tableau
    $T=(R,E)\in\RSE_3(\lambda)$ of level $3$ and shape $\lambda$ on the
    right, where $\lambda=(5,5,5,4,4,2)$ and $\nu=(5,5,5,2,2)$. Note that
    $\row_{\ge3}(R)$ is an SSYT.}
  \label{fig:RSE}
\end{figure}

Note that if $T=(R,E)\in \RSE_1(\lambda)$, then both $R$ and $E$ are SSYTs. Thus
$T$ can be considered as an SSYT whose entries are from
$\{1,2,\dots,1^*,2^*,\dots\}$. Using this observation the following proposition
is easy to verify.

\begin{prop}\label{prop:nc}
  The map $\Tab$ is a weight-preserving bijection between
  $\L_{\lambda/\emptyset}^{\mathrm{NC}}$ and $\RSE_1(\lambda)$.
\end{prop}

If $T=(R,E)\in \RSE_\ell(\lambda)$, then
$E=\emptyset$ and $R$ is an RPP of shape $\lambda$ with no extra conditions.
Hence, we will identify $\RSE_\ell(\lambda)$ with $\RPP(\lambda)$.

The \emph{weight} of $T=(R,E)\in\RSE_{k}(\lambda)$ is defined by
\[
\wt(T) = \wt(R)t_E,
\]
where $t_E=t_1^{c_1(E)}t_2^{c_2(E)}\cdots$ and $c_i(E)$ is the number of $i$'s
in $E$. For example, if $T=(R,E)$ is the RSE-tableau in Figure~\ref{fig:RSE},
then $\wt(R)=x_1^{3}x_2^4x_3^4x_4^2t_1^3t_2^3$, $t_E=t_3^3t_4^2t_5$, and
$\wt(T)=\wt(R)t_E=x_1^{3}x_2^4x_3^4x_4^2t_1^3t_2^3t_3^3t_4^2t_5$.

We now describe two maps $\pu$ and $\pd$ on RSE-tableaux, where the level of an
RSE-tableau is decreased by $\pu$ and increased by $\pd$. These maps are due to
Lam and Pylyavskyy \cite{LP2007} who used them as intermediate steps in their
bijection between $\RPP(\lambda)$ and $\RSE_1(\lambda)$. See
Figures~\ref{fig:pu} and \ref{fig:pd} for illustrations of these maps.

\begin{defn}\label{def:pu}
[The level-decreasing map $\pu:\RSE_{k+1}(\lambda)\to
  \RSE_{k}(\lambda)$] 
  Let $\lambda$ be a partition with $\ell(\lambda)=\ell$ and let $T=(R,E)\in
  \RSE_{k+1}(\lambda)$ with $1\le k\le \ell-1$. Then $\pu(T)$ is defined as
  follows.
\begin{description}
\item[Step 1] For $1\le j\le \lambda_{k}$, the entry $R(k,j)$ is \emph{novel} if
  $j>\lambda_{k+1}$ or $R(k,j)\ne R(k+1,j)$. Let $a_1\le a_2\le \dots\le a_r$ be
  the novel entries. Let $R'$ be the tableau obtained from $R$ by removing row
  $k$ and shifting $\row_{\ge k+1}(R)$ up by one (so that $\row_{\ge
    k}(R')=\row_{\ge k+1}(R)$). Then $H:=\sh(R)/\sh(R')$ is a horizontal strip.
\item[Step 2] Update $R'$ by inserting $a_1,a_2,\dots,a_r$ in this order into
  $\row_{\ge k}(R')$ using the RSK algorithm. By the property of the RSK
  algorithm, if $a_i$ was the novel entry in column $j$, then $a_i$ bumps the
  $(k,j)$-entry of $\row_{\ge k}(R')$ (in case it exists) or $a_i$ is simply
  placed at position $(k,j)$ (in case $\row_{\ge k}(R')$ has no $(k,j)$-entry).
  Therefore row $k$ of $R'$ becomes the original row $k$ of $R$
  and the newly created cells of $R'$ lie in the horizontal strip $H$. Let $E'$
  be the union of $E$ and the remaining empty cells in $H$, which we fill with
  $k$'s. Finally, mark row $k$ of $R'$ as the level and define $\pu(T)=(R',E')$.
\end{description}
\end{defn}

\begin{figure}
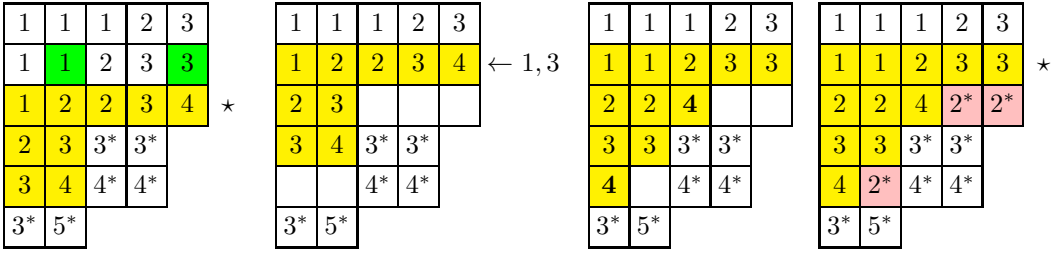

  \centering
  \begin{ytableau}
    1&1&1&2&3\\
    1&*(green)1&2&3&*(green)3\\
    *(yellow)1&*(yellow)2&*(yellow)2&*(yellow)3&*(yellow)4&\none[\star]\\
    *(yellow)2&*(yellow)3&3^*&3^*\\
    *(yellow)3&*(yellow)4&4^*&4^*\\
    3^*&5^*\\
  \end{ytableau}\quad
  \begin{ytableau}
    1&1&1&2&3\\
    *(yellow)1&*(yellow)2&*(yellow)2&*(yellow)3&*(yellow)4&\none[\leftarrow]&\none[ 1,3]\\
    *(yellow)2&*(yellow)3&&&\\
    *(yellow)3&*(yellow)4&3^*&3^*\\
    {}&&4^*&4^*\\
    3^*&5^*\\
  \end{ytableau}\quad
  \begin{ytableau}
    1&1&1&2&3\\
    *(yellow)1&*(yellow)1&*(yellow)2&*(yellow)3&*(yellow)3\\
    *(yellow)2&*(yellow)2&*(yellow)\textbf{4}&&\\
    *(yellow)3&*(yellow)3&3^*&3^*\\
    *(yellow)\textbf{4}&&4^*&4^*\\
    3^*&5^*\\
  \end{ytableau}\quad
  \begin{ytableau}
    1&1&1&2&3\\
    *(yellow)1&*(yellow)1&*(yellow)2&*(yellow)3&*(yellow)3&\none[\star]\\
    *(yellow)2&*(yellow)2&*(yellow)4&*(pink)2^*&*(pink)2^*\\
    *(yellow)3&*(yellow)3&3^*&3^*\\
    *(yellow)4&*(pink)2^*&4^*&4^*\\
    3^*&5^*\\
  \end{ytableau}
  \caption{An illustration of the level-decreasing map $\pu$ applied to
    $T=(R,E)\in\RSE_3(\lambda)$, where $\lambda=(5,5,5,4,4,2)$. The first
    diagram shows $T$ where the novel entries are colored green and $\row_{\ge
      3}(R)$ is colored yellow. If we remove row 2 and shift $\row_{\ge3}(R)$ up
    by one, we get $R'$ in the second diagram. If we insert the novel entries
    $1,3$ into $\row_{\ge2}(R')$, we get the third diagram, where the entries in the newly
    added cells are written in boldface. By filling the empty cells with $2^*$
    we obtain the fourth diagram, which is $\pu(T)\in \RSE_2(\lambda)$.}
  \label{fig:pu}
\end{figure}

\begin{defn} \label{def:pd}
[The level-increasing map $\pd:\RSE_{k}(\lambda)\to \RSE_{k+1}(\lambda)$]
  Let $\lambda$ be a partition with $\ell(\lambda)=\ell$ and let $T=(R,E)\in
  \RSE_{k}(\lambda)$ with $1\le k\le \ell-1$. Then $\pd(T)$ is defined as
  follows.
\begin{description}
\item[Step 1] Let $c_1<c_2<\dots<c_r$ be the column indices $j$ such that column
  $j$ of $E$ does not contain $k$. Let $E'$ be the tableau obtained from $E$ by
  removing the cells containing $k$. For $i=r,r-1,\dots,1$ in this order, apply
  the reverse RSK algorithm to $\row_{\ge k}(R)$ starting from the last cell of
  column $c_i$ and denote the resulting tableau by $R_1$. Let $a_i$ and $b_i$ be
  the integers such that the reverse RSK algorithm bumps $a_i$ at position
  $(k,b_i)$ at the end.
\item[Step 2] Let $R'$ be the tableau obtained from $R$ by replacing $\row_{\ge
    k}(R)$ by $R_1$. Shift $\row_{\ge k}(R')$ down by one so that row $k$ of
  $R'$ is now empty. For each $1\le j\le \lambda_{k}$, if $j=b_i$ for some $i$,
  then let $R'(k,j)=a_i$, and otherwise let $R'(k,j)$ equal $R'(k+1,j)$.
  Finally, mark row $k+1$ of $R'$ as the level and define $\pd(T)=(R',E')$.
\end{description}
\end{defn}

\begin{figure}
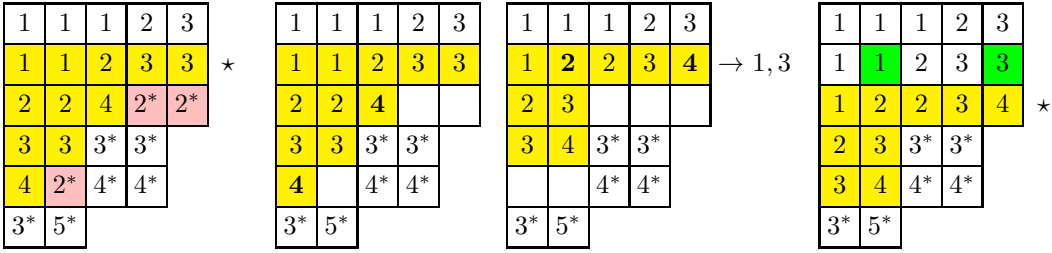

  \begin{ytableau}
    1&1&1&2&3\\
    *(yellow)1&*(yellow)1&*(yellow)2&*(yellow)3&*(yellow)3&\none[\star]\\
    *(yellow)2&*(yellow)2&*(yellow)4&*(pink)2^*&*(pink)2^*\\
    *(yellow)3&*(yellow)3&3^*&3^*\\
    *(yellow)4&*(pink)2^*&4^*&4^*\\
    3^*&5^*\\
  \end{ytableau}\quad
  \begin{ytableau}
    1&1&1&2&3\\
    *(yellow)1&*(yellow)1&*(yellow)2&*(yellow)3&*(yellow)3\\
    *(yellow)2&*(yellow)2&*(yellow)\textbf{4}&&\\
    *(yellow)3&*(yellow)3&3^*&3^*\\
    *(yellow)\textbf{4}&&4^*&4^*\\
    3^*&5^*\\
  \end{ytableau}\quad
  \begin{ytableau}
    1&1&1&2&3\\
    *(yellow)1&*(yellow)\textbf{2}&*(yellow)2&*(yellow)3&*(yellow)\textbf{4}
    &\none[\rightarrow]&\none[1,3]\\
    *(yellow)2&*(yellow)3&&&\\
    *(yellow)3&*(yellow)4&3^*&3^*\\
    {}&&4^*&4^*\\
    3^*&5^*\\
  \end{ytableau}\quad
  \begin{ytableau}
    1&1&1&2&3\\
    1&*(green)1&2&3&*(green)3\\
    *(yellow)1&*(yellow)2&*(yellow)2&*(yellow)3&*(yellow)4&\none[\star]\\
    *(yellow)2&*(yellow)3&3^*&3^*\\
    *(yellow)3&*(yellow)4&4^*&4^*\\
    3^*&5^*\\
  \end{ytableau}
  \caption{An illustration of the level-increasing map $\pd$ applied to
    $T=(R,E)\in\RSE_2(\lambda)$, where $\lambda=(5,5,5,4,4,2)$. The first
    diagram shows $T$, where $\row_{\ge2}(R)$ is colored yellow and the cells
    with a $2^*$ are colored pink. If we remove the cells with a $2^*$ we get
    the second diagram, where the entry of the last cell in each column without
    a $2^*$ is written in boldface. If we perform the reverse RSK algorithm to
    $\row_{\ge2}(R)$ starting from each boldface entry from right to left, we
    obtain $R'$ in the third diagram, where $1$ and $3$ are bumped from columns
    $2$ and $5$ respectively. By shifting $\row_{\ge2}(R')$ down by one, putting
    the bumped entries in the corresponding columns in row $2$, and filling each
    empty cell with the same entry directly below it, we obtain the fourth
    diagram, which is $\pd(T)\in \RSE_3(\lambda)$.}
  \label{fig:pd}
\end{figure}

The following proposition is shown in \cite[Proof of Theorem~9.8]{LP2007}.

\begin{prop}\label{prop:inverse}
  Let $\lambda$ be a partition with $\ell(\lambda)=\ell$. For $1\le k\le \ell-1$,
  the maps $\pd:\RSE_{k}(\lambda)\to \RSE_{k+1}(\lambda)$ and
  $\pu:\RSE_{k+1}(\lambda)\to \RSE_{k}(\lambda)$ are weight-preserving bijections
  and they are mutual inverses.
\end{prop}

As a corollary to Proposition~\ref{prop:inverse}, we obtain that the map
$\pu^{\ell-1}:\RSE_\ell(\lambda)\to\RSE_1(\lambda)$ is a weight-preserving bijection.
Since we can identify $\RSE_\ell(\lambda)$ with $\RPP(\lambda)$, it follows that
\[
  \gg_\lambda(x;t) = \sum_{R\in\RPP(\lambda)} \wt(R)= \sum_{T\in\RSE_\ell(\lambda)}
  \wt(T)= \sum_{T\in\RSE_1(\lambda)} \wt(T).
\]
Note that since $\lnc=\lsnc$ for $\mu=\emptyset$, Proposition~\ref{prop:snc} shows
\[
  \det \left( e_{\lambda'_i-i+j}(x_1,x_2,\dots,t_{1},t_{2},\dots,
    t_{\lambda'_i-1}) \right)_{1\le i,j\le n} = \sum_{\pp\in
    \L^{\mathrm{NC}}_\lambda} \wt(\pp).
\]
By Proposition~\ref{prop:nc},
\[
\sum_{\pp\in \L^{\mathrm{NC}}_\lambda} \wt(\pp) = \sum_{T\in\RSE_1(\lm)}\wt(T).
\]
Combining the above three equations we obtain \eqref{eq:Yeliussizov}. This proof
is essentially the same as Yeliussizov's \cite[\S 10.1]{Yeliussizov2017}.

\section{Skew RSE-tableaux}
\label{sec:skew-rse-tableaux}

In this section we extend the definition of RSE-tableaux to skew shapes and
study properties of the maps $\pu$ and $\pd$ on them. We first need to extend
the definition of elegant tableaux.

Let $\mu$ be a partition. A \emph{$\mu$-elegant tableau} is an SSYT $E$ of a
certain skew shape $\lambda/\nu$ with $\mu\subseteq\nu$ such that $\mu'_j+1\le
E(i,j)\le i-1$ for all $(i,j)\in \lambda/\nu$. Note that the $\emptyset$-elegant
tableaux are the usual elegant tableaux.

\begin{defn}
  A \emph{(skew) RSE-tableau of shape $\lm$ of level $k$} is a pair $T=(R,E)$
  satisfying the following conditions:
\begin{itemize}
\item $R\in\RPP(\nu/\mu)$ with row $k$ of $R$ marked, for some partition $\nu$
  satisfying $\mu\subseteq\nu\subseteq\lambda$ and $\row_{\le k}(\nu)=\row_{\le
    k}(\lambda)$,
\item $\row_{\ge k}(R)$ is an SSYT, and 
\item $E$ is a $\mu$-elegant tableau of shape $\lambda/\nu$ such that every
  entry in $E$ is at least $k$.
\end{itemize}
The set of RSE-tableaux of shape $\lm$ and level $k$ is denoted by
$\RSE_k(\lm)$.
\end{defn}

\begin{figure}
  \centering
  \begin{ytableau}
   \none& \none & \none&\none  &3&3 \\
   \none&\none&\none &1 &3 &\none[\star]\\
   \none & 1& 1&2 \\
   1 &2 &3^* &3^* \\
   2 &4^* \\
  \end{ytableau}\qquad\qquad
  \begin{ytableau}
   \muentry1 & \muentry1 & \muentry1 &\muentry1  &3&3 \\
   \muentry2&\muentry2&\muentry2&1 &3 &\none[\star]\\
   \muentry3& 1& 1&2 \\
   1 &2 &3^* &3^* \\
   2 &4^* \\
  \end{ytableau}
  \caption{A skew RSE-tableau $T\in\RSE_2(\lm)$ for $\lambda=(6,5,4,4,2)$ and
    $\mu=(4,3,1)$ on the left, and its corresponding tableau $\overline{T}\in
    \oRSE_2(\lm)$ on the right, where the cells in $\mu$ are colored gray.}
  \label{fig:skew RSE}
\end{figure}

The \emph{weight} of $T=(R,E)\in\RSE_{k}(\lm)$ is defined in the same way by
\[
\wt(T) = \wt(R)t_E.
\]
For example, if $T$ is the skew RSE-tableau in Figure~\ref{fig:skew RSE}, then
$\wt(R)=x_1^4x_2^3x_3^2t_1$, $t_E=t_3^2t_4$, and
$\wt(T)=\wt(R)t_E=x_1^4x_2^3x_3^2t_1t_3^2t_4$.

We also define $\oRSE_k(\lambda)$
to be the set of RSE-tableaux $(R,E)$ of shape $\lambda$ and level $k$,
where the entries of $R$ are taken from
\begin{equation}
  \label{eq:overline1}
\{\overline{1}<\overline{2}<\cdots<1<2<\cdots\},
\end{equation}
while the entries of $E$ are still positive integers.

For $T=(R,E)\in \RSE_k(\lm)$, let $\overline{T}$ be the RSE-tableau
$(\overline{R},E)\in\oRSE_k(\lambda)$, where $\overline{R}$ is obtained from $R$
by filling the cells in row $i$ of $\mu$ with $\overline{i}$'s for each $1\le
i\le \ell(\mu)$. We will identify $T$ with $\overline{T}$ so that
$\RSE_k(\lm)\subseteq\oRSE_k(\lambda)$. By replacing each entry $i$ in $E$ by
$i^*$ and putting $\overline{R}$ and $E$ together we will also consider
$T=\overline{T}=(\overline{R},E)\in \RSE_k(\lm)$ as a tableau of shape $\lambda$
whose entries are taken from
\begin{equation}
  \label{eq:ov1,1,1^*}
\{\overline{1}<\overline{2}<\cdots<1<2<\cdots<1^*<2^*<\cdots\}.  
\end{equation}
We call the elements in \eqref{eq:ov1,1,1^*} the \emph{extended integers}. See
Figure~\ref{fig:skew RSE} for an example of this correspondence. Sometimes we
will also consider $T\in\RSE_k(\lm)$ as an RPP of shape $\lambda$ whose entries
are extended integers. We call $\overline{i}$ a \emph{negative entry} and $i^*$
an \emph{$\omega$-entry}.

The following proposition is immediate from the definition of $\RSE_k(\lm)$.
\begin{prop}\label{prop:oRSE}
Let $T\in \oRSE_k(\lambda)$. Then $T\in\RSE_k(\lm)$ if and only if the
following conditions hold:
\begin{enumerate}
\item The cells containing a negative entry are exactly those in $\mu$.
\item If $(i,j)\in\mu$, then $T(i,j)=\overline{i}$.
\item If $T(i,j)=a^*$, then $\mu'_{j}+1\le a\le \lambda'_j-1$.
\end{enumerate}
\end{prop}

Note that if $T=(R,E)\in\RSE_1(\lm)$ then we can regard $T$ as an SSYT of shape
$\lm$ whose entries are from $\{1,2,\dots,1^*,2^*,\dots\}$. Using this
observation, similarly to Proposition~\ref{prop:nc}, the following proposition
is easy to verify.
\begin{prop}\label{prop:nc2}
  The map $\Tab$ is a weight-preserving bijection between $\lnc$ and
  $\RSE_1(\lm)$. 
\end{prop}

If $T=(R,E)\in\RSE_\ell(\lm)$, then by definition of an RSE-tableau, we must
have $E=\emptyset$ and $R$ can be any RPP of shape $\lm$ whose entries are
positive integers. Hence, we will identify $\RSE_\ell(\lm)$ with $\RPP(\lm)$.

Using the ordering given by \eqref{eq:overline1}, the same maps $\pu$ and $\pd$
are applied to $\oRSE_k(\lambda)$. By the identification
$\RSE_k(\lm)\subseteq\oRSE_k(\lambda)$, these maps $\pu$ and $\pd$ are also
applied to $\RSE_k(\lm)$. See Figure~\ref{fig:skew pd} for an example.

\begin{figure}
  \centering
  \begin{ytableau}
   \muentry1 & \muentry1 & \muentry1 &\muentry1  &3&3 \\
   \muentry2&\muentry2&\muentry2&1 &3 &\none[\star]\\
   \muentry3& 1& 1&2 \\
   1 &2 &3^* &3^* \\
   2 &4^* \\
  \end{ytableau}\qquad\qquad
  \begin{ytableau}
   \muentry1 & \muentry1 & \muentry1 &\muentry1  &3&3&\none[\star] \\
   \muentry2&\muentry2&\muentry2&1 &1^* \\
   \muentry3& 1& 1&2 \\
   1 &2 &3^* &3^* \\
   2 &4^* \\
  \end{ytableau}
  \caption{If $T_1\in\RSE_2(\lm)$ and $T_2\in\RSE_1(\lm)$ are the left and right
    diagrams respectively, then $\pu(T_1)=T_2$ and $\pd(T_2)=T_1$.}
  \label{fig:skew pd}
\end{figure}

The following definitions will be used frequently for the rest of this paper.
See Figure~\ref{fig:le} for an example. 

\begin{defn}
  Let $T_1$ and $T_2$ be RPPs whose entries are extended integers. Define
  $T_1\sqcup T_2$ to be the tableau obtained by concatenating $T_1$ and $T_2$,
  i.e., $\col_{\le k}(T_1\sqcup T_2)=T_1$ and $\col_{\ge k+1}(T_1\sqcup
  T_2)=T_2$, where $k$ is the number of columns in $T_1$. Define $T_1\le T_2$ if
  $T_1\sqcup T_2$ is also an RPP (with extended integers).
\end{defn}

\begin{figure}
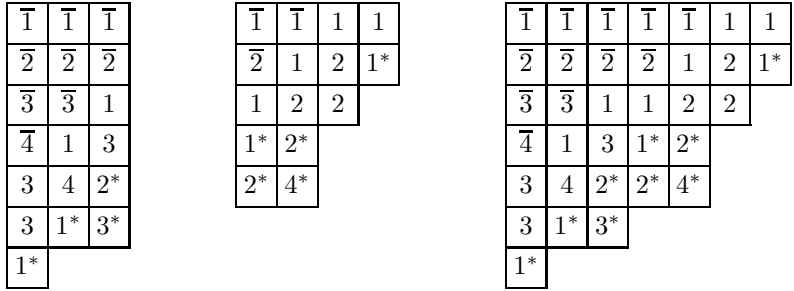

  \centering
  \begin{ytableau}
    \overline{1}&\overline{1}&\overline{1}\\
    \overline{2}&\overline{2}&\overline{2}\\
    \overline{3}&\overline{3}&1\\
    \overline{4}&1&3\\
    3&4&2^*\\
    3&1^*&3^*\\
    1^*\\
  \end{ytableau}\qquad\qquad
  \begin{ytableau}
    \overline{1}&\overline{1}&1&1\\
    \overline{2}&1&2&1^*\\
    1&2&2\\
    1^*&2^*\\
    2^*&4^*\\
    \none\\
    \none
  \end{ytableau}\qquad\qquad
  \begin{ytableau}
    \overline{1}&\overline{1}&\overline{1}&\overline{1}&\overline{1}&1&1\\
    \overline{2}&\overline{2}&\overline{2}&\overline{2}&1&2&1^*\\
    \overline{3}&\overline{3}&1&1&2&2\\
    \overline{4}&1&3&1^*&2^*\\
    3&4&2^*&2^*&4^*\\
    3&1^*&3^*\\
    1^*\\
  \end{ytableau}
  \caption{An RPP $T_1$ on the left, an RPP $T_2$ in the middle, and $T_1\sqcup
    T_2$ on the right. Since $T_1\sqcup T_2$ is also an RPP, we have $T_1\le
    T_2$.}
  \label{fig:le}
\end{figure}

Note that if $T_1$ and $T_2$ are RSE-tableaux of levels $k_1$ and $k_2$,
respectively, with $k_1\le k_2$ such that $T_1\le T_2$, then $T_1\sqcup T_2$
with row $k_2$ marked is an RSE-tableau of level $k_2$.

In Lemma~\ref{lem:pu} below we will show that $\pu$ is a map from $\RSE_k(\lm)$
to $\RSE_{k-1}(\lm)$, i.e., for every $T\in \RSE_k(\lm)$ we have
$\pu(T)\in\RSE_{k-1}(\lm)$. On the contrary, $\pd$ does not always send an
element in $\RSE_{k-1}(\lm)$ to an element in $\RSE_{k}(\lm)$, see
Figure~\ref{fig:not in image}. We will find equivalent conditions for
$T\in\RSE_{k-1}(\lm)$ to satisfy $\pd(T)\in \RSE_{k}(\lm)$ in
Lemma~\ref{lem:pd}.

\begin{figure}
  \centering
  \begin{ytableau}
   \overline{1} & 1&\none[\star] \\
   2 &1^* \\
  \end{ytableau}\qquad\qquad
  \begin{ytableau}
   \overline{1} & 1 \\
   \overline{1} &2 &\none[\star]\\
  \end{ytableau}
  \caption{An RSE-tableau $T\in\RSE_1((2,2)/(1))$ on the left and its image
    $\pd(T)\in\oRSE_2((2,2))$ on the right. Note that
    $\pd(T)\not\in\RSE_2((2,2)/(1))$.}
  \label{fig:not in image}
\end{figure}

\begin{lem}\label{lem:pu}
  Let $T\in\RSE_{k+1}(\lm)$. Then $\pu(T)\in \RSE_{k}(\lm)$ and, for all $1\le
  s\le \mu_k$, 
  \begin{equation}
    \label{eq:pu col_i}
    \pu(T) = \col_{\le s}(T) \sqcup \pu(\col_{\ge s+1}(T)).    
  \end{equation}
\end{lem}
\begin{proof}
  Let $T=(R,E)$. Since $T\in\RSE_{k+1}(\lm)\subseteq\oRSE_{k+1}(\lambda)$, we
  have $\pu(T)\in\oRSE_{k}(\lambda)$. In order to show $\pu(T)\in
  \RSE_{k}(\lm)$, we must show that $\pu(T)$ satisfies the three conditions in
  Proposition~\ref{prop:oRSE}. To this end we first prove the following claim,
  which is equivalent to \eqref{eq:pu col_i}.

  \textbf{Claim:} If $1\le s\le\mu_k$, then $\col_{\le s}(\pu(T))= \col_{\le
    s}(T)$ and $\col_{\ge s+1}(\pu(T)) = \pu(\col_{\ge s+1}(T))$.

  The first $s$ entries of row $k$ in $T$ are all $\overline{k}$ and every entry
  in row $k+1$ is either $\overline{k+1}$ or a positive integer. Thus the first
  $s$ entries are novel entries. In the definition of $\pu$ we delete row $k$ of
  $R$ and insert the novel entries into $\row_{\ge k+1}(R)$ (after shifting it
  up by one). Since $\overline{k}$ is smaller than every entry in $\row_{\ge
    k+1}(R)$, each of the first $s$ insertion paths is a straight vertical path.
  Since insertion paths never intersect, the first $s$ columns are not changed
  after the insertion of the first $s$ $\overline{k}$'s. This shows the first
  identity of the claim. The fact that the insertion paths starting from columns
  of index greater than $\mu_k$ never enter columns of index at least $\mu_k$
  also implies the second identity of the claim.

  We now show that $\pu(T)$ satisfies the three conditions in
  Proposition~\ref{prop:oRSE}. For the first two conditions it is enough to show
  that the restrictions of $T$ and $\pu(T)$ to $\mu$ are equal because $\pu$
  preserves the total number of negative entries. Since $\mu$ is contained in
  $\row_{\le k}(\lambda)\cup \col_{\le \mu_k}(\lambda)$, this follows from the
  special case $\col_{\le \mu_k}(\pu(T))= \col_{\le \mu_k}(T)$ of the claim and
  the fact $\row_{\le k}(\pu(T))=\row_{\le k}(T)$. For the third condition note
  that in the construction of $\pu(T)=(R',E')$ from $T=(R,E)$, $E'$ is obtained
  from $E$ by adding some $k^*$'s. Suppose $\pu(T)(i,j)=a^*$. If $a\ne k$, then
  we must have $T(i,j)=a^*$. Since $T\in\RSE_{k+1}(\lm)$, in this case
  $\mu'_j+1\le a\le \lambda'_j-1$. If $a=k$, we must have $j\ge \mu_k+1$ since
  $\col_{\le \mu_k}(\pu(T))=\col_{\le \mu_k}(T)$ and $T$ has no $k^*$. But $j\ge
  \mu_k+1$ implies that $\mu'_j<k$, and $\pu(T)\in\oRSE_{k}(\lm)$ implies $k\le
  \lambda'_j-1$. Hence the third condition also holds and the proof is
  completed.
\end{proof}

\begin{lem}\label{lem:pd}
  Let $T\in\RSE_{k}(\lm)$. Then the following are equivalent:
  \begin{enumerate}
  \item $\pd(T)\in \RSE_{k+1}(\lm)$,
  \item $T\in \pu(\RSE_{k+1}(\lm))$,
  \item $\pd(T) = \col_{\le s}(T) \sqcup \pd(\col_{\ge s+1}(T))$, for all $1\le
    s\le \mu_k$, and
  \item $\col_{\le \mu_k}(T) \le \pd(\col_{\ge \mu_k+1}(T))$.
  \end{enumerate}
\end{lem}
\begin{proof}
We will prove the implications $(1)\Rightarrow (2) \Rightarrow (3)\Rightarrow
(4)\Rightarrow (1)$. 

$(1)\Rightarrow (2)$: Let $T'=\pd(T)\in\RSE_{k+1}(\lm)$. Then $T=\pu(T')\in
\pu(\RSE_{k+1}(\lm))$.

$(2)\Rightarrow (3)$: Suppose $T=\pu(T')$ for some $T'\in\RSE_{k+1}(\lm)$. Then
$T' = \pd(T)$. We need to show that for $1\le s\le \mu_k$, 
\begin{align*}
  \col_{\le s}(T') &= \col_{\le s}(T),\\ 
  \col_{\ge s+1}(T') &= \pd(\col_{\ge s+1}(T)).
\end{align*}
By Lemma~\ref{lem:pu}, $T=\pu(T') = \col_{\le s}(T') \sqcup \pu(\col_{\ge
  s+1}(T'))$. This shows that $\col_{\le s}(T)=\col_{\le s}(T')$, which
is the first equality, and $\col_{\ge s+1}(T)=\pu(\col_{\ge s+1}(T'))$,
which is equivalent to the second equality after applying $\pd$. 

$(3)\Rightarrow (4)$: The fact that $\pd(T) = \col_{\le \mu_k}(T) \sqcup
\pd(\col_{\ge \mu_k+1}(T))$ is an RPP (with extended integers as entries) shows
that $\col_{\le \mu_k}(T) \le \pd(\col_{\ge \mu_k+1}(T))$.

$(4)\Rightarrow (1)$: Suppose that $T=(R,E)\in\RSE_{k}(\lm)$ satisfies
$\col_{\le \mu_k}(T) \le \pd(\col_{\ge \mu_k+1}(T))$. To show $\pd(T)\in
\RSE_{k+1}(\lm)$ we must show that $\pd(T)$ satisfies the three conditions in
Proposition~\ref{prop:oRSE}. Since the restriction of $\pd(T)$ to the
$\omega$-entries is exactly the same as that of $T$ with $k^*$ deleted and
$T\in\RSE_{k+1}(\lm)$ satisfies the third condition, so does $\pd(T)$. For the
first two conditions, it is enough to show that the restrictions of $T$ and
$\pd(T)$ to $\mu$ are the same. Since $\mu$ is contained in $\row_{\le
  k}(\lambda) \cup \col_{\ge\mu_k+1}(\lambda)$ and $\row_{\le
  k}(\pd(T))=\row_{\le k}(T)$, it suffices to prove the following equality:
\begin{equation}
  \label{eq:2}
\col_{\le \mu_k}(\pd(T))= \col_{\le \mu_k}(T).  
\end{equation}

To show \eqref{eq:2} we investigate the construction of $\pd(T)$ in
Definition~\ref{def:pd}. Let $c_1<c_2<\dots<c_r$ be the column indices $j$ such
that column $j$ of $E$ does not contain $k$. Since $(R,E)\in\RSE_k(\lm)$, $E$ is
a $\mu$-elegant tableau. Thus for every cell $(i,j)$ of $E$ with $1\le j\le
\mu_k$ we have $E(i,j)\ge\mu'_j+1\ge k+1$. This shows that $E$ has no entries
equal to $k$ in the first $\mu_k$ columns, i.e., $c_j=j$ for all $1\le j\le
\mu_k$.

Recall that in the definition of $\pd(T)$ we apply the reverse RSK algorithm to
$\row_{\ge k}(R)$ starting from the last cell of column $c_j$ for
$j=r,r-1,\dots,1$ in this order. We denote by $P_j$ each inverse bumping path.

\textbf{Claim}: $P_{\mu_k},P_{\mu_k-1},\dots,P_1$ are straight vertical paths.

By the construction of $\pd(T)$ the claim implies \eqref{eq:2}. Hence it suffices
to prove the claim. To this end let $Q$ be the tableau obtained from $\row_{\ge
  k}(R)$ by applying the reverse RSK algorithm to the last cell of column $c_j$
for $j=r,r-1,\dots,\mu_{k}+1$. Note that the same process is applied to
$\row_{\ge k}(\col_{\ge \mu_k+1}(T))$ when we compute $(R_1,E_1)=\pd(\col_{\ge
  \mu_k+1}(T))$. Since $\row_{\ge k+1}(R_1)$ has been shifted down in Step 2 of
the definition of $\pd$, we have
\begin{equation}
  \label{eq:3}
 \row_{\ge k+1}(R_1) = \row_{\ge k}(Q). 
\end{equation}

Suppose that $R_1(k,\mu_k+1)<R_1(k+1,\mu_k+1)$. This means that the entry
$R_1(k,\mu_k+1)$ was bumped during the applications of the reverse RSK
algorithm. Since column $\mu_k+1$ is the leftmost nonempty column of $\col_{\ge
  \mu_k+1}(T)$, the reverse bumping path that pushed this cell must be a
straight vertical path and we must have $c_{\mu_k+1}=\mu_k+1$. Since each path
$P_j$, for $1\le j\le \mu_k+1$, starts from column $j$ and $P_{\mu_{k}+1}$ is a
straight vertical path, by the non-intersecting property of the reverse bumping
paths, the claim follows.

Suppose now that $R_1(k,\mu_k+1)=R_1(k+1,\mu_k+1)$. By the same reasoning as in
the previous paragraph, it is sufficient to show that $P_{\mu_k}$ is a straight
vertical path. By the assumption $\col_{\le \mu_k}(T) \le \pd(\col_{\ge
  \mu_k+1}(T))$, we have $\col_{\le \mu_k}(R)\le R_1$. This together with
\eqref{eq:3} implies that $R(i,\mu_k)\le R_1(i,\mu_k+1) = Q(i-1,\mu_k+1)$ for
all $i\ge k+1$. These inequalities and the assumption
$R_1(k,\mu_k+1)=R_1(k+1,\mu_k+1)$ ensure that $P_{\mu_k}$ is a straight vertical
path. This completes the proof of the claim.
\end{proof}

The following two lemmas will be useful in the next section. 

\begin{lem}\label{lem:pd2}
  Suppose $T\in\RSE_{b}(\col_{\ge c+1}(V/\mu))$, where $b,c$ are nonnegative
  integers, $V$ is a vertical diagram, and $\mu$ is a partition with
  $\mu\subseteq V$ such that $\col_{\ge c+1}(V/\mu)$ is a skew shape and $1\le
  c\le\mu_b$. Then for all $1\le d\le b-1$ and $1\le j\le \mu_{b-1}$, we have
  \begin{align}
    \label{eq:pddT}
    \pu^d(T)&\in \RSE_{b-d}(\col_{\ge c+1}(V/\mu)),\\
    \label{eq:pd col_i}
    \pu^d(T) &= \col_{\le j}(T) \sqcup \pu^d(\col_{\ge j+1}(T)).        
  \end{align}
  In other words \eqref{eq:pd col_i} means that applying $\pu^d$ to $T$ is the same
  thing as applying $\pu^d$ only to the columns of index greater than $j$
  while keeping $\col_{\le j}(T)$ unmodified.
\end{lem}
\begin{proof}
  Since $\col_{\ge c+1}(V/\mu)$ is a skew shape, applying Lemma~\ref{lem:pu}
  repeatedly gives \eqref{eq:pddT}.

  We prove \eqref{eq:pd col_i} by induction on $d$. If $d=1$, it is just
  Lemma~\ref{lem:pu}. Suppose that \eqref{eq:pd col_i} is true for $1\le d\le
  b-2$ and consider the $d+1$ case. Using Lemma~\ref{lem:pu} with
  \eqref{eq:pddT} and the inequality $j\le\mu_{b-1}\le \mu_{b-d-1}$, we have
  \begin{equation}
    \label{eq:10}
\pu(\pu^d(T)) = \col_{\le j}(\pu^d(T))\sqcup \pu(\col_{\ge j+1}(\pu^d(T))).     
  \end{equation}
Note that the induction hypothesis \eqref{eq:pd col_i} for the $d$ case is
equivalent to
\[
\col_{\le j}(\pu^d(T)) = \col_{\le j}(T), \qquad
\col_{\ge j+1}(\pu^d(T)) = \pu^d(\col_{\ge j+1}(T)). 
\]
Hence \eqref{eq:10} can be written as
\[
  \pu^{d+1}(T) = \col_{\le j}(T)\sqcup \pu(\pu^d(\col_{\ge j+1}(T)))
  = \col_{\le j}(T)\sqcup \pu^{d+1}(\col_{\ge j+1}(T)),
\]
which is \eqref{eq:pd col_i} for the $d+1$ case. 
This completes the proof. 
\end{proof}

\begin{lem}\label{lem:pu2}
  Suppose $T\in\pu^a(\RSE_{b}(\col_{\ge c+1}(V/\mu))$, where $a,b,c$ are
  nonnegative integers, $V$ is a vertical diagram, and $\mu$ is a partition with
  $\mu\subseteq V$ such that $\col_{\ge c+1}(V/\mu)$ is a skew shape and $1\le
  c\le\mu_b$. Then for all $1\le d\le a$ and $1\le j\le \mu_{b-a+d-1}$, we have
  \begin{align}
    \label{eq:11}
  \pd^d(T)&\in\RSE_{b-a+d}(\col_{\ge c+1}(V/\mu)),\\
    \label{eq:pu col_i2}
    \pd^d(T) &= \col_{\le j}(T) \sqcup \pd^d(\col_{\ge j+1}(T)).    
  \end{align}
  In other words \eqref{eq:pu col_i2} means that applying $\pd^d$ to $T$ is the same
  thing as applying $\pd^d$ only to the columns of index greater than $j$
  while keeping $\col_{\le j}(T)$ unmodified.
\end{lem}
\begin{proof}
  Since $\pu$ and $\pd$ are inverses of each other, the assumption
  $T\in\pu^a(\RSE_{b}(\col_{\ge c+1}(V/\mu))$ together with Lemma~\ref{lem:pd2}
  implies that for all $1\le d\le a$,
  \[
    \pd^d(T)\in \pu^{a-d}(\RSE_{b}(\col_{\ge c+1}(V/\mu))
    \subseteq\RSE_{b-a+d}(\col_{\ge c+1}(V/\mu)),
  \]
  which shows \eqref{eq:11}. For the second statement, let $T'=\pd^d(T)$. Then
  \eqref{eq:pu col_i2} is equivalent to
  \begin{equation}
    \label{eq:12}
    \col_{\le j}(T') = \col_{\le j}(T),\qquad
    \col_{\ge j+1}(T') = \pd^d(\col_{\ge j+1}(T)).
  \end{equation}
Since $T'=\pd^d(T)\in\RSE_{b-a+d}(\col_{\ge c+1}(V/\mu))$, by Lemma~\ref{lem:pd2},
\[
T=\pu^d(T') = \col_{\le j}(T')\sqcup \pu^d(\col_{\ge j+1}(T')). 
\]
This shows that
\[
    \col_{\le j}(T) = \col_{\le j}(T'),\qquad
    \col_{\ge j+1}(T) = \pu^d(\col_{\ge j+1}(T')).
\]
By taking $\pd^d$ in each side of the second equation we obtain \eqref{eq:12},
completing the proof.
\end{proof}

Recall that at the end of the previous section we showed that
$\pu^{\ell-1}:\RSE_\ell(\lambda)\to\RSE_1(\lambda)$ is a weight-preserving
bijection and 
\[
  \gg_\lambda(x;t) = \sum_{R\in\RPP(\lambda)} \wt(R)=
  \sum_{T\in\RSE_\ell(\lambda)} \wt(T)= \sum_{T\in\RSE_1(\lambda)} \wt(T).
\]
For the skew shape case the map $\pu:\RSE_k(\lm)\to \RSE_{k-1}(\lm)$ is not a
bijection but just an injection.

\begin{prop}\label{prop:wpinj}
  Let $\lambda$ and $\mu$ be partitions with $\mu\subseteq\lambda$ and
  $\ell(\lambda)=\ell$. Then, for $2\le k\le \ell$, the map $\pu:\RSE_k(\lm)\to
  \RSE_{k-1}(\lm)$ is a weight-preserving injection. In other words,
  $\pu:\RSE_k(\lm)\to \pu(\RSE_k(\lm))$ is a weight-preserving bijection.
\end{prop}
\begin{proof}
  This follows from Proposition~\ref{prop:inverse} and Lemma~\ref{lem:pu}.
\end{proof}

The fact that $\pu:\RSE_k(\lm)\to \RSE_{k-1}(\lm)$ is an injection can still be
used to give a different expression for $\gg_\lm(x;t)$.

\begin{prop}\label{prop:glm1}
  Let $\lambda$ and $\mu$ be partitions with $\mu\subseteq\lambda$ and
  $\ell(\lambda)=\ell$. Then 
\[
\gg_\lm(x;t)  = \sum_{T\in \pu^{\ell-1}(\RSE_\ell(\lm))}\wt(T).
\]
\end{prop}
\begin{proof}
  By the identification of $\RPP(\lm)$ and $\RSE_\ell(\lm)$,
  \[
\gg_\lm(x;t) = \sum_{R\in \RPP(\lm)}\wt(R)
= \sum_{T\in \RSE_\ell(\lm)}\wt(T).
\]
Applying Proposition~\ref{prop:wpinj} repeatedly we obtain that
$\pu^{\ell-1}:\RSE_\ell(\lm)\to \pu^{\ell-1}(\RSE_\ell(\lm))$ is a weight-preserving
bijection. Therefore
\[
  \sum_{T\in \RSE_\ell(\lm)}\wt(T) =\sum_{T\in \pu^{\ell-1}(\RSE_\ell(\lm))}\wt(T) ,
\]
and the proof follows. 
\end{proof}

By Propositions~\ref{prop:snc} and \ref{prop:glm1}, in order to prove
Theorem~\ref{thm:main} it is sufficient to show the following proposition whose
proof will be given in the next section.

\begin{prop}\label{prop:snc=pu}
  Let $\lambda$ and $\mu$ be partitions with $\mu\subseteq\lambda$, $\ell(\lambda)=
  \ell$, and $\ell(\lambda')=n$. Then
\[
    \sum_{\pp\in \lsnc} \wt(\pp) = \sum_{T\in \pu^{n-1}(\RSE_\ell(\lm))}\wt(T).
  \]
\end{prop}

\section{Sign-reversing involution}
\label{sec:sign-revers-invol}

In this section we define a sign-reversing involution on $\lsnc$ to prove
Proposition~\ref{prop:snc=pu}. Recall Notation~\ref{notation}. For any tableau
$Q$ denote by $\col_{D_k}(Q)$ the part of $Q$ consisting of column $j$ for all
$j\in D_k$. The following proposition is an immediate consequence of
Proposition~\ref{prop:nc2}.

\begin{prop}\label{prop:snc RSE1}
  Let $\pp=(p_1,\dots,p_n)\in\LL$ and $T=\Tab(\pp)$. Then $\pp\in\lsnc$ if and
  only if each $\col_{D_k}(T)$ (with row $1$ marked) is an RSE-tableau of level
  $1$ (and of some skew shape).
\end{prop}

We now define the sign-reversing involution $\Phi$ on $\lsnc$. See
Section~\ref{sec:example-phi} for a concrete example of the map $\Phi$.

\begin{defn}[The sign-reversing involution $\Phi$ on $\lsnc$]
  \label{def:Phi}
  Let $\pp\in\lsnc$. Then $\Phi(\pp)$ is defined as follows. Here we use the
  letters defined in Notation~\ref{notation}.
  \begin{description}
  \item[Step 1] Suppose $T=\Tab(\pp)\in \VT(\pi(\lambda)/\mu)$ and write
    $T=T_{r+1}\sqcup \dots\sqcup T_2\sqcup T_1$, where each $T_i=\col_{D_i}T$ is
    considered as an RSE-tableau of level $1$.
  \item[Step 2] Let $U_1=T_1$. For $i=1,2,\dots,r$, if $U_{i}$ has been
    defined and $T_{i+1}\le \pd^{m_i}(U_{i})$, define $U_{i+1}$ to be the
    RSE-tableau $T_{i+1}\sqcup \pd^{m_i}(U_{i})$ with level $M_i+1$.
  \item[Step 3] If $U_{r+1}$ is defined, set $\Phi(\pp)=\pp$. Otherwise, let $k$ be
    the smallest integer such that $T_{k+1}\not\le \pd^{m_k}(U_{k})$. In order
    to define $\Phi(\pp)$ we proceed as follows.
    \begin{description}
    \item[Step 3-1] Let $\gamma=(\gamma_1,\dots,\gamma_\ell)$ be the partition
      defined by
      \[
\gamma_i = 
        \begin{cases}
          \lambda_i, & \mbox{if $1\le i \le M_k$,}\\
          \min(\lambda_i,d_k), & \mbox{if $M_k+1\le i\le \ell$}.
        \end{cases}
      \]      
     Considering $\row_{\ge
        M_k+1}(T_{k+1}\sqcup \pd^{m_k}(U_{k}))$ as an element in
      $\VT(\pi(\lambda)/\gamma)$, let $\mathbf{q}=(q_1,\dots,q_n)$ be the $n$-path such
      that
\[
  \Tab(\mathbf{q})=\row_{\ge M_k+1}(T_{k+1}\sqcup \pd^{m_k}(U_{k})).
\]
\item[Step 3-2] Let $s$ be the largest integer such that column $s$ of
  $\row_{\ge M_k+1}(\pd^{m_k}(U_{k}))$ is nonempty. Choose the intersection
  point $(a,b)$ of $q_i$ and $q_j$ for $d_{k+1}+1\le i, j\le s$ in such a way
  that $(b,a)$ is the largest in the lexicographic order. Let $q'_i$ and $q'_j$
  be the paths obtained from $q_i$ and $q_j$, respectively, by exchanging their
  subpaths after the intersection $(a,b)$. Define $\mathbf{q'}$ to be the
  $n$-path $\mathbf{q}$ in which $q_i$ and $q_j$ are replaced by $q'_i$ and
  $q'_j$, respectively.
\item[Step 3-3] Note that $\row_{D_k}(\Tab(\mathbf{q}))=\row_{\ge
    M_k+1}(T_{k+1})$ and $\row_{\ge d_k+1}(\Tab(\mathbf{q}))=\row_{\ge
    M_k+1}(\pd^{m_k}(U_{k}))$. Let $\widetilde{T}_{k+1}$ be the RSE-tableau of
  level $1$ obtained from $T_{k+1}$ by replacing $\row_{D_k}(\Tab(\mathbf{q}))$
  by $\row_{D_k}(\Tab(\mathbf{q'}))$. Let $\widetilde{U}_{k}$ be the RSE-tableau
  of level $M_k+1$ obtained from $\pd^{m_k}(U_k)$ by replacing replacing
  $\row_{\ge d_k+1}(\Tab(\mathbf{q}))$ by $\row_{\ge d_k+1}(\Tab(\mathbf{q'}))$.
  See Figure~\ref{fig:qq'} for an illustration of the construction of
  $\widetilde{T}_{k+1}$ and $\widetilde{U}_{k}$. Let
  \[
    T'=T_r\sqcup \cdots \sqcup T_{k+2}\sqcup \widetilde{T}_{k+1}\sqcup
    \pu^{M_k}(\widetilde{U}_{k}).
  \]
  Finally, define $\Phi(\pp)$ to be the $n$-path $\pp'$ satisfying
  $\Tab(\pp') = T'$. 
    \end{description}
  \end{description}
\end{defn}

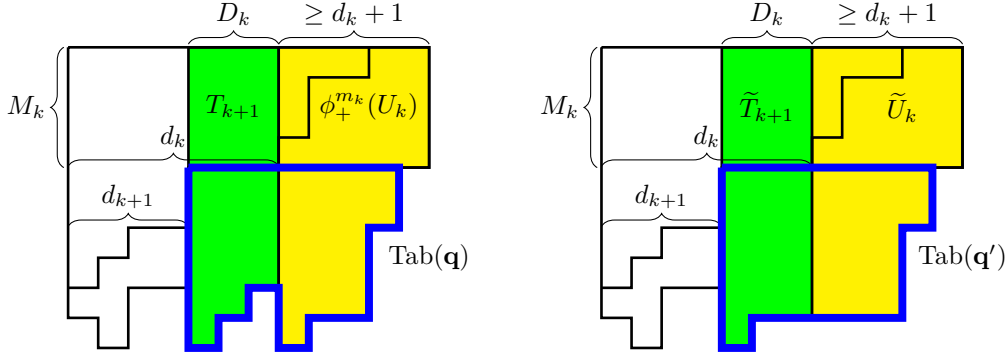
\begin{figure}
  \centering
   \begin{tikzpicture}[scale=0.4](0,0)(10,11)
\draw[line width=1pt,fill=green] (4,10)--++(0,-10)--++(1,0)--++(0,1)--++(1,0)--++(0,1)
--++(1,0)--++(0,8)--++(-3,0);
\draw[line width=1pt,fill=yellow] (7,10)--++(0,-10)--++(1,0)--++(0,1)
--++(2,0)--++(0,3)--++(1,0)--++(0,2)--++(1,0)--++(0,4)--++(-5,0);

\draw[line width=1pt](0,10)--++(0,-9)--++(1,0)
--++(0,-1)--++(1,0)--++(0,2)--++(2,0)--++(0,-2)--++(1,0)--++(0,1)--++(1,0)--++(0,1)
--++(1,0)--++(0,-2)--++(1,0)--++(0,1)--++(2,0)--++(0,3)--++(1,0)
--++(0,2)--++(1,0)--++(0,4)--++(-12,0);

\draw[line width=1pt](0,2)--++(1,0)--++(0,1)--++(1,0)--++(0,1)--++(2,0)
--++(0,2)--++(3,0)--++(0,1)--++(1,0)--++(0,2)--++(2,0)--++(0,1);

\draw [decorate,decoration={brace,amplitude=5pt},xshift=-4pt,yshift=0pt]
(0,6) -- (0,10) node [midway,xshift=-0.5cm] {$M_k$}; 
\draw [decorate,decoration={brace,amplitude=5pt},yshift=1pt,yshift=0pt]
(0,6) -- (7,6) node [midway,yshift=0.4cm] {$d_{k}$}; 
\draw [decorate,decoration={brace,amplitude=5pt},yshift=1pt,yshift=0pt]
(0,4) -- (4,4) node [midway,yshift=0.4cm] {$d_{k+1}$}; 

\draw [decorate,decoration={brace,amplitude=5pt},yshift=3pt]
(4,10) -- (7,10) node [midway,yshift=0.4cm] {$D_k$}; 
\draw [decorate,decoration={brace,amplitude=5pt},yshift=3pt]
(7,10) -- (12,10) node [midway,yshift=0.4cm] {$\ge d_k+1$}; 

\node at (5.5,8) {$T_{k+1}$};
\node at (10,8) {$\pd^{m_k}(U_{k})$};
\node at (12,3) {$\Tab(\mathbf{q})$};

\draw[blue,line width=3pt] (4,6)--++(0,-6)--++(1,0)--++(0,1)
--++(1,0)--++(0,1)--++(1,0)
--++(0,-2)--++(1,0)--++(0,1)--++(2,0)--++(0,3)--++(1,0)--++(0,2)--++(-7,0);
  \end{tikzpicture}\qquad
   \begin{tikzpicture}[scale=0.4](0,0)(10,11)
\draw[line width=1pt,fill=green] (4,10)--++(0,-10)--++(1,0)--++(0,1)--++(2,0)
--++(0,9)--++(-3,0);
\draw[line width=1pt,fill=yellow] (7,10)--++(0,-9)
--++(3,0)--++(0,3)--++(1,0)--++(0,2)--++(1,0)--++(0,4)--++(-5,0);

\draw[line width=1pt](0,10)--++(0,-9)--++(1,0)
--++(0,-1)--++(1,0)--++(0,2)--++(2,0)--++(0,-2)--++(1,0)--++(0,1)
--++(5,0)--++(0,3)--++(1,0)
--++(0,2)--++(1,0)--++(0,4)--++(-12,0);

\draw[line width=1pt](0,2)--++(1,0)--++(0,1)--++(1,0)--++(0,1)--++(2,0)
--++(0,2)--++(3,0)--++(0,1)--++(1,0)--++(0,2)--++(2,0)--++(0,1);

\draw [decorate,decoration={brace,amplitude=5pt},xshift=-4pt,yshift=0pt]
(0,6) -- (0,10) node [midway,xshift=-0.5cm] {$M_k$}; 
\draw [decorate,decoration={brace,amplitude=5pt},yshift=1pt,yshift=0pt]
(0,6) -- (7,6) node [midway,yshift=0.4cm] {$d_{k}$}; 
\draw [decorate,decoration={brace,amplitude=5pt},yshift=1pt,yshift=0pt]
(0,4) -- (4,4) node [midway,yshift=0.4cm] {$d_{k+1}$}; 

\draw [decorate,decoration={brace,amplitude=5pt},yshift=3pt]
(4,10) -- (7,10) node [midway,yshift=0.4cm] {$D_k$}; 
\draw [decorate,decoration={brace,amplitude=5pt},yshift=3pt]
(7,10) -- (12,10) node [midway,yshift=0.4cm] {$\ge d_k+1$}; 

\node at (5.5,8) {$\widetilde{T}_{k+1}$};
\node at (10,8) {$\widetilde{U}_{k}$};
\node at (12,3) {$\Tab(\mathbf{q'})$};

\draw[blue,line width=3pt] (4,6)--++(0,-6)--++(1,0)--++(0,1)
--++(5,0)--++(0,3)--++(1,0)--++(0,2)--++(-7,0);
  \end{tikzpicture}
  \caption{The construction of $\widetilde{T}_{k+1}$ and $\widetilde{U}_{k}$.
    The tableau $\widetilde{T}_{k+1}\sqcup\widetilde{U}_{k}$ is obtained from
    $T_{k+1}\sqcup\pd^{m_k}(U_k)$ by replacing $\Tab(\mathbf{q})$ by
    $\Tab(\mathbf{q'})$.}
  \label{fig:qq'}
\end{figure}

The main theorem in this section is as follows.

\begin{thm}\label{thm:involution}
  The map $\Phi$ is a sign-reversing involution on $\lsnc$ whose fixed point set is
  \[
\left\{\pp\in\lsnc: \Tab(\pp)\in \pu^{\ell-1}(\RSE_\ell(\lm))\right\}.
  \]
\end{thm}

Note that Theorem~\ref{thm:involution} immediately implies
Proposition~\ref{prop:snc=pu}, and hence completes the proof of
Theorem~\ref{thm:main}. The rest of this section is devoted to proving
Theorem~\ref{thm:involution}. We will constantly use the notation in
Definition~\ref{def:Phi}.

We first show that $\Phi$ is a well-defined map on $\lsnc$. The only thing that
needs to be checked is Step~3-3 in the construction of $\Phi(\pp)$. More
precisely we must check the three assertions in the following lemma. One can see
that these three assertions imply $\Phi(\pp)=\pp'\in\lsnc$ as follows. By the
second assertion, we obtain that $\pu^{M_k}(\widetilde{U}_{k})$ is an
RSE-tableau of level $1$. This together with the first assertion implies that
$\col_{D_k}(T')$ is an RSE-tableau of level $1$ for all $1\le k\le r+1$. Then
Propositions~\ref{prop:Tab}, \ref{prop:snc RSE1} and the third assertion
imply that $\pp'\in\lsnc$.

\begin{lem}\label{lem:sign-reversing}
  Let $\pp\in\lsnc$ with $\type(\pp)=\pi$. Suppose that $U_{r+1}$ is not defined
  in the construction of $\Phi(\pp)$. Then
\begin{enumerate}
\item $\widetilde{T}_{k+1}$ is an RSE-tableau of level $1$, 
\item $\widetilde{U}_{k}$ is an RSE-tableau of level $M_{k}+1$, and
\item $T'\in\VT(\pi'(\lambda)/\mu)$,
where $\pi'=\pi (i,j)$ and  $(i,j)$ is the transposition exchanging
  $i$ and $j$.
\end{enumerate}

\end{lem}
\begin{proof}
  Recall that $\mathbf{q}'=(q'_1,\dots,q'_n)$ and $\mathbf{q}=(q_1,\dots,q_n)$
  differ only by the $i$th and $j$th paths, where $q'_i$ and $q'_j$ are obtained
  from $q_i$ and $q_j$ by exchanging the subpaths after the intersection
  $(a,b)$. Suppose $i<j$ so that $i\in D_k$ and $j\ge d_k+1$. The choice of the
  intersection point $(a,b)$ in Step 3-2 guarantees that both $\{q'_l:
  d_{k+1}+1\le l\le d_k\}$ and $\{q'_l: d_{k}+1\le l\le s\}$ are
  non-intersecting. This implies that $\col_{D_k}(\Tab(\mathbf{q}'))$ and
  $\col_{\ge d_k+1}(\Tab(\mathbf{q}'))$ are SSYTs (whose entries are extended
  integers). Moreover, since $i<j$, the initial point $(\mu'_j+n-j,0)$ of $q_j$
  is to the left of the initial point $(\mu'_i+n-i,0)$ of $q_i$. Therefore the
  intersection $(a,b)$ of $q_i$ and $q_j$ must occur after the first diagonal
  step of $q_j$. This shows that row $M_k+1$ of $\col_{\ge
    d_k+1}(\Tab(\mathbf{q'}))$ is the same as that of $\col_{\ge
    d_k+1}(\Tab(\mathbf{q}))$.

 Note that 
 \begin{align*}
   \row_{\le M_k}(\widetilde{T}_{k+1})&=\row_{\le M_k}(T_{k+1}),\\
   \row_{\le M_k}(\widetilde{U}_{k})&=\row_{\le M_k}(\pd^{m_k}(U_k)),\\
   \row_{\ge M_k+1}(\widetilde{T}_{k+1})&=\col_{D_k}(\Tab(\mathbf{q}')),\\
   \row_{\ge M_k+1}(\widetilde{U}_{k})&=\col_{\ge d_k+1}(\Tab(\mathbf{q}')),
 \end{align*}
 where everything in the right-hand side is an SSYT (whose entries are extended
 integers) except $\row_{\le M_k}(\pd^{m_k}(U_k))$, which is an RPP. Thus
 checking the first and second assertions reduces to checking the following:
  \begin{enumerate}
  \item Rows $M_k$ and $M_{k}+1$ of $\widetilde{T}_{k+1}$ form an SSYT.
  \item Rows $M_k$ and $M_{k}+1$ of $\widetilde{U}_{k}$ form an RPP.
  \end{enumerate}
  The first statement is true because row $M_k$ of $\widetilde{T}_{k+1}$ is
  empty (or equivalently filled with negative entries $\overline{M_k}$'s). The
  second statement is also true because rows $M_k$ and $M_k+1$ of
  $\widetilde{U}_{k}$ are identical with those of $\col_{\ge
    d_k+1}(\pu^{m_k}(U_k))$ by the last sentence of the previous paragraph.

  For the final assertion recall that $\Tab(\pp)\in\VT(\pi(\lambda)/\mu)$ and
  $\Tab(\mathbf{q})\in\VT(\pi(\lambda)/\gamma)$. By Lemma~\ref{lem:VT},
  $\type(\pp)=\type(\mathbf{q})=\pi$. Since $\mathbf{q'}$ is obtained from
  $\mathbf{q}$ by changing the tails of $q_i$ and $q_j$, we have
  $\type(\mathbf{q}')=\pi'$. Hence, by Lemma~\ref{lem:VT},
  $\Tab(\mathbf{q}')\in\VT(\pi'(\lambda)/\gamma)$. Since $T'$ has the same inner
  shape as $T$ and the same outer shape as $\Tab(\mathbf{q}')$, the third
  assertion follows.
\end{proof}

The following lemma gives a more direct way of computing $U_i$.

\begin{lem}\label{lem:U_i}
  Using the notation in Definition~\ref{def:Phi}, for $1\le i\le r+1$, if
  $U_{i}$ is defined, then 
\[
\pu^{M_{i-1}}(U_i)=\col_{\ge d_i+1}(T),
\]
or equivalently,
\[
U_i=\pd^{M_{i-1}}(\col_{\ge d_i+1}(T)),
\]
where $M_0=0$. 
\end{lem}
\begin{proof}
  We proceed by induction on $i$. It is true for the case $i=1$, which is
  $U_1=T_1$. Assume true for the case $i\ge1$ and consider the case $i+1$. Since
  $U_{i+1}\in\RSE_{M_{i}+1}(\col_{\ge d_{i+1}+1}(\pi(\lambda)/\mu))$ and
  $\mu_{M_{i}}=d_{i}$, by Lemma~\ref{lem:pu2},
  $\pu^{m_i}(U_{i+1})\in\RSE_{M_{i-1}+1}(\col_{\ge d_{i+1}+1}(\pi(\lambda)/\mu))$ and
\[
\pu^{m_i}(U_{i+1})=\col_{\le d_i}(U_{i+1})\sqcup \pu^{m_i}(\col_{\ge d_{i}+1}(U_{i+1})).
\]
Using the construction of $U_{i+1}=T_{i+1}\sqcup\pd^{m_i}(U_i)$
the above equation can be rewritten as
\[
  \pu^{m_i}(U_{i+1})=T_{i+1}\sqcup \pu^{m_i}(\pd^{m_i}(U_i))
  =T_{i+1}\sqcup U_i,
\]
which implies
\begin{equation}
  \label{eq:8}
\col_{\le d_i}(\pu^{m_i}(U_{i+1}))=T_{i+1},\qquad  \col_{\ge d_{i}+1}(\pu^{m_i}(U_{i+1}))=U_i.
\end{equation}
Since $\pu^{m_i}(U_{i+1})\in\RSE_{M_{i-1}+1}(\col_{\ge d_{i+1}+1}(\lm))$ and
$d_i=\mu_{M_i}<\mu_{M_{i-1}}$, by using Lemma~\ref{lem:pu2} and \eqref{eq:8}
 we obtain
\[
\pu^{M_{i-1}} (\pu^{m_i}(U_{i+1}))
=\col_{\le d_i}(\pu^{m_i}(U_{i+1}))\sqcup \pu^{M_{i-1}}(\col_{\ge d_{i}+1}(\pu^{m_i}(U_{i+1})))
=T_{i+1}\sqcup \pu^{M_{i-1}}(U_i).
\]
By the induction hypothesis, $\pu^{M_{i-1}}(U_i)= \col_{\ge d_{i}+1}(T)$.
Hence the above equation can be rewritten as
\[
\pu^{M_{i}} (U_{i+1})=T_{i+1}\sqcup \col_{\ge d_{i}+1}(T)= \col_{\ge d_{i+1}+1}(T),
\]
which is the desired statement for the $i+1$ case. This completes the proof.
\end{proof}

The following lemma gives an equivalent condition for $U_i$ to be defined.

\begin{lem}\label{lem:U_i2}
  Using the notation in Definition~\ref{def:Phi}, for $1\le i\le r+1$, $U_{i}$ is
  defined if and only if
  \begin{equation}
    \label{eq:7}
  \col_{\ge d_i+1}(T)\in \pu^{M_{i-1}}(\RSE_{M_{i-1}+1}(\col_{\ge d_i+1}(\rho))).    
  \end{equation}
where $\rho=\sh(T)=\pi(\lambda)/\mu$. 
\end{lem}
\begin{proof}
  Suppose that $U_i$ is defined. Then \eqref{eq:7} follows immediately from
  Lemma~\ref{lem:U_i} since $U_{i}\in\RSE_{M_{i-1}+1}(\col_{\ge d_i+1}(\rho))$.

  Conversely suppose that \eqref{eq:7} holds. We will prove by induction that
  $U_s$ is defined for all $1\le s\le i$. The case $s=1$ is trivial. Assume that
  $U_{s}$ is defined and $1\le s\le i-1$. Then by Lemma~\ref{lem:U_i},
  \begin{equation}
    \label{eq:9}
U_s=\pd^{M_{s-1}}(\col_{\ge d_s+1}(T)).
  \end{equation}

  Recall that $U_{s+1}$ is defined if $T_{s+1}\le\pd^{m_s}(U_s)$. To show this
  let $Q=\col_{\ge d_i+1}(T)$. Since
  $Q\in\pu^{M_{i-1}}(\RSE_{M_{i-1}+1}(\col_{\ge d_i+1}(\rho)))$, by
  Lemma~\ref{lem:pu2} with $a=M_{i-1}, b=M_{i-1}+1, d=M_{s-1}$, and
  $j=\mu_{b-a-d-1}=\mu_{M_{s-1}}=d_s$, we have
  \begin{equation}
    \label{eq:4}
    \pd^{M_s}(Q) = \col_{\le d_s}(Q)\sqcup \pd^{M_s}(\col_{\ge d_s+1}(Q))
  \in \RSE_{M_s+1}(\col_{\ge c+1}(\rho)).
  \end{equation}
  This implies that $\col_{\le d_s}(Q)\le \pd^{M_s}(\col_{\ge d_s+1}(Q))$.
  Since the leftmost columns of $\col_{\le d_s}(Q)$ and $T_{s+1}$ coincide, we
  also have
\[
T_{s+1}\le \pd^{M_s}(\col_{\ge d_s+1}(Q)).
\]
 On the other hand, by \eqref{eq:9}, 
\[
  \pd^{M_s}(\col_{\ge d_s+1}(Q)) = \pd^{M_s}(\col_{\ge d_s+1}(T)) =
  \pd^{m_s}(U_s).
\]
The above two equations show that $T_{s+1}\le \pd^{m_s}(U_s)$ and hence
$U_{s+1}$ is defined. Therefore by induction $U_i$ is also defined, which
completes the proof.
\end{proof}

The following lemma shows that $\Phi$ has the desired fixed points. 

\begin{lem}\label{lem:fixed points}
We have $\Phi(\pp)=\pp$ if and only if $T=\Tab(\pp)\in \pu^{\ell-1}(\RSE_\ell(\lm))$.
\end{lem}
\begin{proof}
Suppose $\Phi(\pp)=\pp$. Then $U_{r+1}$ is defined and therefore by Lemma~\ref{lem:U_i},
\begin{equation}
  \label{eq:6}
  T=\col_{\ge d_{r+1}+1}(T)\in \pu^{M_{r}}(\RSE_{M_{r}+1}(\lm)).  
\end{equation}
Thus $T=\pu^{M_{r}}(Q)$ for some $Q\in\RSE_{M_{r}+1}(\lm)$. To obtain $T\in
\pu^{\ell-1}(\RSE_\ell(\lm))$ it is enough to show that
$Q\in\pu^{m_{r+1}-1}(\RSE_\ell(\lm))$ because it would imply that there is
$Q'\in\RSE_\ell(\lm)$ satisfying
\[
T=\pu^{M_{r}}(Q)=\pu^{M_{r}}(\pu^{m_{r+1}-1}(Q')) = \pu^{\ell-1}(Q'). 
\]
Since $Q=\pd^{m_{r+1}-1}(Q')$ if and only if $\pu^{m_{r+1}-1}(Q)=Q'$, the
condition $Q\in\pu^{m_{r+1}-1}(\RSE_\ell(\lm))$ is equivalent to
$\pd^{m_{r+1}-1}(Q)\in\RSE_\ell(\lm)$. We will show by induction that 
\begin{equation}
  \label{eq:5}
  \pd^{i}(Q)\in\RSE_{M_r+1+i}(\lm), \qquad  0\le i\le m_{r+1}-1, 
\end{equation}
which is true for $i=0$ by assumption. Let $0\le i\le m_{k+1}-2$ and suppose
\eqref{eq:5} is true for $i$. Now we apply Lemma~\ref{lem:pd} with
$T=\pd^{i}(Q)$ and $k=M_r+1+i$. Since $\mu_{M_r+1+i}=0$, the fourth condition
of the lemma trivially holds. Hence the first condition of the lemma also holds
and we obtain $\pd(\pd^{i}(Q))\in\RSE_{M_r+1+i+1}(\lm)$, which is exactly
\eqref{eq:5} with $i+1$. By induction \eqref{eq:5} is true for all $0\le i\le
m_{r+1}-1$, and in particular we obtain $\pd^{m_{r+1}-1}(Q)\in\RSE_\ell(\lm)$
as desired.

Conversely, suppose that $T\in \pu^{\ell-1}(\RSE_\ell(\lm))$. To show
$\Phi(\pp)=\pp$, we must show that $U_{r+1}$ is defined, which is by
Lemma~\ref{lem:U_i} equivalent to \eqref{eq:6}. By the assumption there is
$Q\in\RSE_\ell(\lm)$ with $T=\pu^{\ell-1}(Q)$. Let $Q'=\pu^{m_{r+1}-1}(Q)$
so that
\[
T=\pu^{\ell-1}(Q)= \pu^{M_{r}}(\pu^{m_{r+1}-1}(Q))=
\pu^{M_{r}}(Q'). 
\]
By applying Lemma~\ref{lem:pu} repeatedly, we obtain $Q'=\pu^{m_{r+1}-1}(Q)\in
\RSE_{M_r+1}(\lm)$, which shows \eqref{eq:6}.
\end{proof}

The following lemma shows that $\Phi$ is indeed an involution.

\begin{lem}\label{lem:involution}
 If $\Phi(\pp)=\pp'$ and $\pp\ne\pp'$, then $\Phi(\pp')=\pp$.
\end{lem}
\begin{proof}
  In this proof we will use the notation $X(\pp)$ to indicate the object $X$,
  for example $X=U_i$ or $X=\widetilde{U}_k$, in Definition~\ref{def:Phi} when we apply
  $\Phi$ to $\pp$.

  Let $\rho=\sh(T(\pp'))$. Then
  $\widetilde{U}_k(\pp)\in\RSE_{M_{k}+1}(\col_{\ge d_k+1}(\rho))$, and by
  Lemma~\ref{lem:pu}, we have
  $\pu^{m_k}(\widetilde{U}_k(\pp))\in\RSE_{M_{k-1}+1}(\col_{\ge d_k+1}(\rho))$.
  Since
  \[
    \col_{\ge d_k+1}(T(\pp'))=\pu^{M_k}(\widetilde{U}_{k}(\pp))
    =\pu^{M_{k-1}}(\pu^{m_k}(\widetilde{U}_{k}(\pp))) \in
     \pu^{M_{k-1}}(\RSE_{M_{k-1}+1}(\col_{\ge d_k+1}(\rho))),
  \]
  by Lemma~\ref{lem:U_i2}, $U_k(\pp')$ is defined. Then by Lemma~\ref{lem:U_i},
  \[
    U_k(\pp')= \pd^{M_{k-1}}(\col_{\ge d_k+1}(T(\pp')))
    =\pd^{M_{k-1}}(\pu^{M_k}(\widetilde{U}_{k}(\pp)))=\pu^{m_k}(\widetilde{U}_k(\pp)),
  \]
  or equivalently,
  \[
 \widetilde{U}_k(\pp) = \pd^{m_k}(U_k(\pp')).
  \]
  By the construction we have $\widetilde{T}_{k+1}(\pp)\not\le
  \widetilde{U}_k(\pp)$. Since $T_{k+1}(\pp')=\widetilde{T}_{k+1}(\pp)$, we obtain
\[
T_{k+1}(\pp') \not\le \pd^{m_k}(U_k(\pp')).
\]

Therefore in the construction of $\Phi(\pp')$,
$U_1(\pp'),U_2(\pp'),\dots,U_k(\pp')$ are defined but not $U_{k+1}(\pp')$.
Observe that in Step 3-1 we compute 
\[
\Tab(\mathbf{q}(\pp')) = \row_{\ge M_k+1}(T_{k+1}(\pp')\sqcup \pd^{m_k}(U_k(\pp'))) =
  \row_{\ge M_k+1}(\widetilde{T}_{k+1}(\pp)\sqcup \widetilde{U}_k(\pp)))
=\Tab(\mathbf{q'}(\pp)).
\]
 Since the map in Step 3-2 sending $\mathbf{q}$ to
$\mathbf{q}'$ is easily seen to be an involution, we obtain that
$\Tab(\mathbf{q'}(\pp')) = \Tab(\mathbf{q}(\pp))$ and therefore 
\[
  \widetilde{T}_{k+1}(\pp') = T_{k+1}(\pp), \qquad
  \widetilde{U}_{k}(\pp') = \pd^{m_k}(U_k(\pp)).
\]
Then
\begin{align*}
  T'(\pp')&=T_r(\pp')\sqcup \cdots \sqcup T_{k+2}(\pp')\sqcup 
            \widetilde{T}_{k+1}(\pp')\sqcup \pu^{M_k}(\widetilde{U}_{k}(\pp'))\\
          &=T_r(\pp)\sqcup \cdots \sqcup T_{k+2}(\pp)\sqcup 
            T_{k+1}(\pp)\sqcup \pu^{M_{k-1}}(U_{k}(\pp)).
\end{align*}
By Lemma~\ref{lem:U_i}, $\pu^{M_{k-1}}(U_{k}(\pp))=\col_{\ge d_k+1}(T(\pp))$,
and we obtain $T'(\pp')=T(\pp)$. This means that $\Phi(\pp')=\pp$ as desired.
\end{proof}

Now we can prove Theorem~\ref{thm:involution} easily. By Lemmas~\ref{lem:fixed
  points} and \ref{lem:involution}, $\Phi$ is an involution on $\lsnc$ with the
desired fixed point set. Suppose that $\Phi(\pp)=\pp'$ and $\pp\ne\pp'$.
Lemma~\ref{lem:VT} and the third assertion of Lemma~\ref{lem:sign-reversing}
show that $\sign(\type(\pp')) = -\sign(\type(\pp))$. Note that
$\wt(\pp)=\sign(\type(\pp))\wt(T)$ and $\wt(\pp')=\sign(\type(\pp'))\wt(T')$. By
the construction of $\Phi$, we have $\wt(T)=\wt(T')$, and hence
$\wt(\pp')=-\wt(\pp)$, which completes the proof.

\section{An example of the involution $\Phi$}
\label{sec:example-phi}

In this section we give a concrete example of the involution $\Phi$ applied to
$\pp$ which is not a fixed point.

  Let $n=6$, $\lambda=(6,6,5,5,5,5,5,5,4)$ and $\mu=(5,3,3,1,1,1)$. Then
  $\lambda'=(9,9,9,9,8,2)$ and $\mu'=(6,3,3,1,1)$. Let $A_i=(\mu'_i+n-1,0)$ and
  $B_i=(\lambda'_i+n-i,2\omega)$. Consider the $n$-path $\pp\in \lsnc$ in
  Figure~\ref{fig:ppT}. Note that $\type(\pp)=\pi=(3,2,4,1,5,6)$. We construct
  $\Phi(\pp)=\pp'$ as follows.

 \begin{figure}
   \centering
\begin{tikzpicture}[x=0.4cm, y=0.4cm]
\RSEgrid{15}96
\node[below] at (0,0){$A_6$};
\node[below] at (2,0){$A_5$};
\node[below] at (3,0){$A_4$};
\node[below] at (6,0){$A_3$};
\node[below] at (7,0){$A_2$};
\node[below] at (11,0){$A_1$};
\node[above] at (2,19){$B_6$};
\node[above] at (9,19){$B_5$};
\node[above] at (11,19){$B_4$};
\node[above] at (12,19){$B_3$};
\node[above] at (13,19){$B_2$};
\node[above] at (14,19){$B_1$};
\filldraw (0,0) circle [radius=2pt];
\filldraw (2,0) circle [radius=2pt];
\filldraw (3,0) circle [radius=2pt];
\filldraw (6,0) circle [radius=2pt];
\filldraw (7,0) circle [radius=2pt];
\filldraw (11,0) circle [radius=2pt];
\filldraw (2,19) circle [radius=2pt];
\filldraw (9,19) circle [radius=2pt];
\filldraw (11,19) circle [radius=2pt];
\filldraw (12,19) circle [radius=2pt];
\filldraw (13,19) circle [radius=2pt];
\filldraw (14,19) circle [radius=2pt];

  \draw[line width=2pt] (0,0)--
  node[above]{1} ++(1,1)--
  node[above]{2} ++(1,1)--
  (2,19);

  \draw[line width=2pt,red] (2,0)--
  ++(0,1)--
  node[above]{2} ++(1,1)--
  node[above]{3} ++(1,1)--
  ++(0,3)--
  node[above]{7} ++(1,1)--
  node[above]{8} ++(1,1)--
  ++(0,4)--
  node[above]{$2^*$} ++(1,1)--
  node[above]{$3^*$} ++(1,1)--
  ++(0,1)--
  node[above]{$5^*$} ++(1,1)--
  (9,19);

  \draw[line width=2pt, red] (3,0)--
  ++(0,1)--
  node[below]{2} ++(1,1)--
  node[above]{3} ++(1,1)--
  node[above]{4} ++(1,1)--
  node[above]{5} ++(1,1)--
  node[above]{6} ++(1,1)--
  node[above]{7} ++(1,1)--
  node[above]{8} ++(1,1)--
  node[above]{9} ++(1,1)--
  ++(0,4)--
  node[above]{$3^*$} ++(1,1)--
  node[above]{$4^*$} ++(1,1)--
  node[above]{$5^*$} ++(1,1)--
  (14,19);
  
  \draw[line width=2pt,dashed] (7,0)--
  ++(0,2)--
  node[above]{$3$} ++(1,1)--
  ++(0,2)--
  node[below]{$6$} ++(1,1)--
  ++(0,1)--
  node[below]{$8$} ++(1,1)--
  ++(0,6)--
  node[above]{$4^*$} ++(1,1)--
  node[above]{$5^*$} ++(1,1)--
  node[above]{$6^*$} ++(1,1)--
  (13,19);

  \draw[line width=2pt,dashed] (6,0)--
  ++(0,2)--
  node[above]{$3$} ++(1,1)--
  ++(0,4)--
  node[above]{$8$} ++(1,1)--
  node[above]{$9$} ++(1,1)--
  ++(0,5)--
  node[above]{$4^*$} ++(1,1)--
  ++(0,1)--
  node[above]{$6^*$} ++(1,1)--
  (11,19); 

  \draw[line width=2pt, blue] (11,0)--
  ++(0,1)--
  node[above,black]{$2$} ++(1,1)--
  (12,19);
\end{tikzpicture}\qquad \qquad
   \begin{ytableau}
     \none[T_4] & \none[T_3] & \none[T_3] & \none[T_2] & \none[T_2] & \none[T_1]\\
     \muentry1 & \muentry1& \muentry1& \muentry1 & \muentry1 & 1 \\
     \muentry2 & \muentry2 & \muentry2& 2 & 2 & 2\\
     \muentry3 & \muentry3 & \muentry3 & 3 & 3\\
     \muentry4 & 3 & 3 & 4 & 7 \\
     \muentry5 & 6 & 8 & 5 & 8 \\
     \muentry6 & 8 & 9 & 6 & 2^* \\
     2 & 4^* & 4^* & 7 & 3^* \\
     \none & 5^* & 6^* & 8 & 5^*\\
     \none & 6^* & \none & 9 \\
     \none & \none & \none & 3^* \\
     \none & \none & \none & 4^* \\
     \none & \none & \none & 5^* \\
     \none
   \end{ytableau} 
   \caption{An $n$-path $\pp\in \lsnc$ on the left and the vertical tableau
     $T=\Tab(\pp)\in \VT(\pi(\lambda)/\mu)$ on the right, where $n=6$,
     $\lambda=(6,6,5,5,5,5,5,5,4)$, $\mu=(5,3,3,1,1,1)$, and
     $\pi=(3,2,4,1,5,6)$. Each $p_i$ is a path from $A_i=(\mu'_i+n-i,0)$ to
     $B_{\pi_i}=(\lambda'_{\pi_i}+n-\pi_i)$. The gray cells are those in $\mu$.
     At the top of each column is written the tableau $T_i$ containing that
     column.}
   \label{fig:ppT}
 \end{figure}
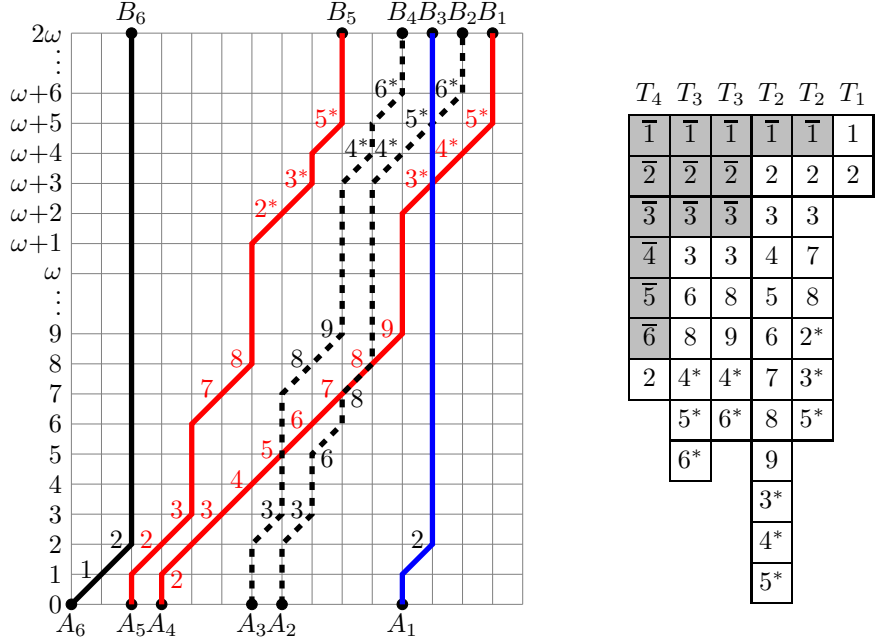

  In Step 1, we find $T=\Tab(\pp)$ and express $T=T_4\sqcup T_3 \sqcup T_2\sqcup
  T_1$ as shown in Figure~\ref{fig:ppT}. Then $T\in\VT(\pi(\lambda)/\mu)$.

  \begin{figure}
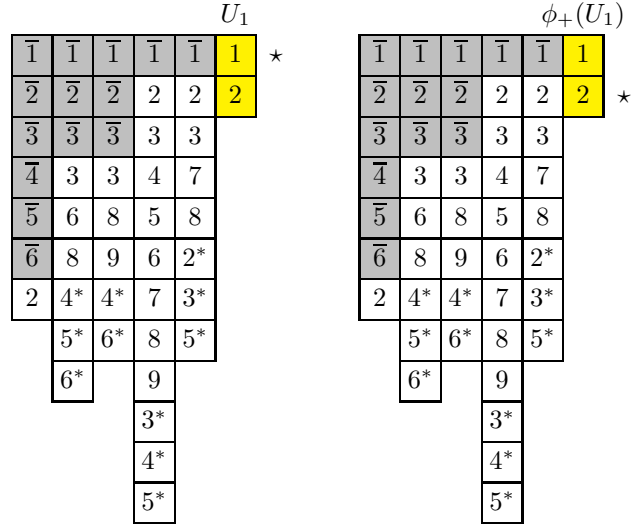

    \centering
   \begin{ytableau}
     \none & \none & \none& \none& \none & \none[U_1]\\
     \muentry1 & \muentry1& \muentry1& \muentry1 & \muentry1 & *(yellow)1 &\none[\star]\\
     \muentry2 & \muentry2 & \muentry2& 2 & 2 & *(yellow)2\\
     \muentry3 & \muentry3 & \muentry3 & 3 & 3\\
     \muentry4 & 3 & 3 & 4 & 7 \\
     \muentry5 & 6 & 8 & 5 & 8 \\
     \muentry6 & 8 & 9 & 6 & 2^* \\
     2 & 4^* & 4^* & 7 & 3^* \\
     \none & 5^* & 6^* & 8 & 5^*\\
     \none & 6^* & \none & 9 \\
     \none & \none & \none & 3^* \\
     \none & \none & \none & 4^* \\
     \none & \none & \none & 5^*
   \end{ytableau} \qquad
   \begin{ytableau}
     \none & \none & \none& \none& \none & \none[\pd(U_1)]\\
     \muentry1 & \muentry1& \muentry1& \muentry1 & \muentry1 & *(yellow)1 \\
     \muentry2 & \muentry2 & \muentry2& 2 & 2 & *(yellow)2&\none[\star]\\
     \muentry3 & \muentry3 & \muentry3 & 3 & 3\\
     \muentry4 & 3 & 3 & 4 & 7 \\
     \muentry5 & 6 & 8 & 5 & 8 \\
     \muentry6 & 8 & 9 & 6 & 2^* \\
     2 & 4^* & 4^* & 7 & 3^* \\
     \none & 5^* & 6^* & 8 & 5^*\\
     \none & 6^* & \none & 9 \\
     \none & \none & \none & 3^* \\
     \none & \none & \none & 4^* \\
     \none & \none & \none & 5^* 
   \end{ytableau}
    \caption{$U_1=T_1$ is the last column in the left diagram and
     $\pd^{m_1}(U_1)=\pd(U_1)$ is the last column in the right diagram. The
     level of each RSE-tableau is marked by a star.}
    \label{fig:U0}
  \end{figure}

  \begin{figure}
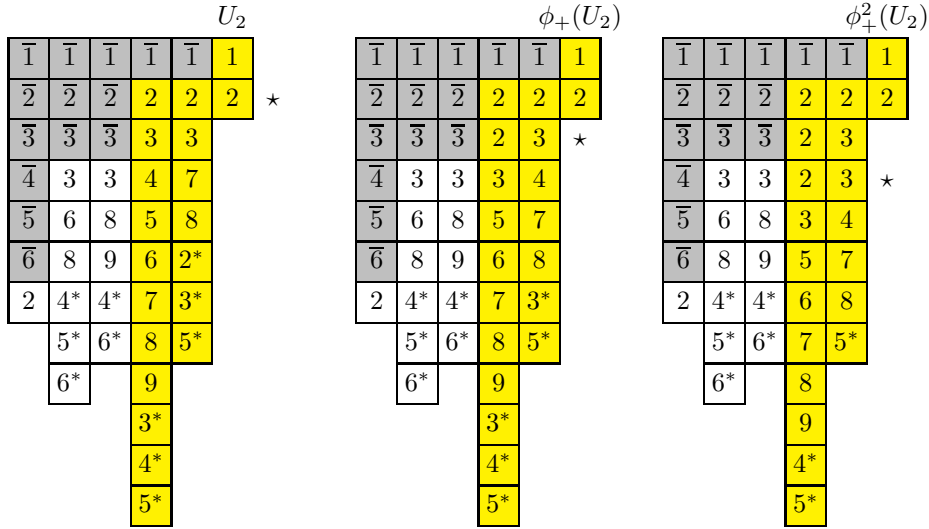

   \centering
   \begin{ytableau}
     \none & \none & \none& \none& \none & \none[U_2]\\
     \muentry1 & \muentry1& \muentry1& \muentry1 & \muentry1 & *(yellow)1 \\
     \muentry2 & \muentry2 & \muentry2& *(yellow)2 & *(yellow)2 & *(yellow)2&\none[\star]\\
     \muentry3 & \muentry3 & \muentry3 & *(yellow)3 & *(yellow)3\\
     \muentry4 & 3 & 3 & *(yellow)4 & *(yellow)7 \\
     \muentry5 & 6 & 8 & *(yellow)5 & *(yellow)8 \\
     \muentry6 & 8 & 9 & *(yellow)6 & *(yellow)2^* \\
     2 & 4^* & 4^* & *(yellow)7 & *(yellow)3^* \\
     \none & 5^* & 6^* & *(yellow)8 & *(yellow)5^*\\
     \none & 6^* & \none & *(yellow)9 \\
     \none & \none & \none & *(yellow)3^* \\
     \none & \none & \none & *(yellow)4^* \\
     \none & \none & \none & *(yellow)5^* \\
   \end{ytableau} \qquad
   \begin{ytableau}
     \none & \none & \none& \none& \none & \none[\pd(U_2)]\\
     \muentry1 & \muentry1& \muentry1& \muentry1 & \muentry1 & *(yellow)1 \\
     \muentry2 & \muentry2 & \muentry2& *(yellow)2 & *(yellow)2 & *(yellow)2\\
     \muentry3 & \muentry3 & \muentry3 & *(yellow)2 & *(yellow)3&\none[\star]\\
     \muentry4 & 3 & 3 & *(yellow)3 & *(yellow)4 \\
     \muentry5 & 6 & 8 & *(yellow)5 & *(yellow)7 \\
     \muentry6 & 8 & 9 & *(yellow)6 & *(yellow)8 \\
     2 & 4^* & 4^* & *(yellow)7 & *(yellow)3^* \\
     \none & 5^* & 6^* & *(yellow)8 & *(yellow)5^*\\
     \none & 6^* & \none & *(yellow)9 \\
     \none & \none & \none & *(yellow)3^* \\
     \none & \none & \none & *(yellow)4^* \\
     \none & \none & \none & *(yellow)5^* \\
   \end{ytableau} \qquad
   \begin{ytableau}
     \none & \none & \none& \none& \none & \none[\pd^2(U_2)]\\
     \muentry1 & \muentry1& \muentry1& \muentry1 & \muentry1 & *(yellow)1 \\
     \muentry2 & \muentry2 & \muentry2& *(yellow)2 & *(yellow)2 & *(yellow)2\\
     \muentry3 & \muentry3 & \muentry3 & *(yellow)2 & *(yellow)3\\
     \muentry4 & 3 & 3 & *(yellow)2 & *(yellow)3 &\none[\star]\\
     \muentry5 & 6 & 8 & *(yellow)3 & *(yellow)4 \\
     \muentry6 & 8 & 9 & *(yellow)5 & *(yellow)7 \\
     2 & 4^* & 4^* & *(yellow)6 & *(yellow)8 \\
     \none & 5^* & 6^* & *(yellow)7 & *(yellow)5^*\\
     \none & 6^* & \none & *(yellow)8 \\
     \none & \none & \none & *(yellow)9 \\
     \none & \none & \none & *(yellow)4^* \\
     \none & \none & \none & *(yellow)5^* \\
   \end{ytableau} 
   \caption{$U_2=T_2\sqcup \pd^{m_1}(U_1)$ is the last three columns in the left
     diagram, $\pd(U_2)$ is the last three columns in the middle diagram, and
     $\pd^{m_2}(U_2)=\pd^{2}(U_2)$ is the last three columns in the right
     diagram.}
   \label{fig:Phi}
 \end{figure}

  In Step 2, we find $U_1$ and $\pd^{m_1}(U_1)=\pd(U_1)$ as in
  Figure~\ref{fig:U0}, and $U_2=T_2\sqcup \pd^{m_1}(U_1)$ and $\pd^{m_2}(U_2)$
  as in Figure~\ref{fig:Phi}. Observe that $T_3$, which is columns $2$ and $3$
  in the right diagram of Figure~\ref{fig:Phi}, does not satisfy $T_3\le
  \pd^{m_2}(U_2)$. Therefore $U_3$ is not defined and Step 2 is finished.

  In Step 3, $k=2$ is the smallest integer such that $T_{k+1}\not\le
  \pd^{m_k}(U_{k})$.

 \begin{figure}
   \centering
   \begin{ytableau}
    \muentry{}&  \muentry{}&\muentry{}&\muentry{}&\muentry{}&\muentry{}\\ 
    \muentry{}&  \muentry{}&\muentry{}&\muentry{}&\muentry{}&\muentry{}\\ 
    \muentry{}&  \muentry{}&\muentry{}&\muentry{}&\muentry{}\\ 
     \muentry{} & *(yellow)3 & *(yellow)3 & *(yellow)2 & *(yellow)3 \\
     \muentry{} & *(yellow)6 & *(yellow)8 & *(yellow)3 & *(yellow)4 \\
     \muentry{} & *(yellow)8 & *(yellow)9 & *(yellow)5 & *(yellow)7 \\
     \muentry{} & *(yellow)4^* & *(yellow)4^* & *(yellow)6 & *(yellow)8 \\
     \none & *(yellow)5^* & *(yellow)6^* & *(yellow)7 & *(yellow)5^*\\
     \none & *(yellow)6^* & \none & *(yellow)8 \\
     \none & \none & \none & *(yellow)9 \\
     \none & \none & \none & *(yellow)4^* \\
     \none & \none & \none & *(yellow)5^* \\
     \none
   \end{ytableau} \qquad
\begin{tikzpicture}[x=0.4cm, y=0.4cm]
\RSEgrid{15}96
\node[below] at (2,0){$C_6$};
\node[below] at (4,0){$C_5$};
\node[below] at (5,0){$C_4$};
\node[below] at (6,0){$C_3$};
\node[below] at (7,0){$C_2$};
\node[below] at (12,0){$C_1$};
\node[above] at (2,19){$B_6$};
\node[above] at (9,19){$B_5$};
\node[above] at (11,19){$B_4$};
\node[above] at (12,19){$B_3$};
\node[above] at (13,19){$B_2$};
\node[above] at (14,19){$B_1$};
\filldraw (2,0) circle [radius=2pt];
\filldraw (4,0) circle [radius=2pt];
\filldraw (5,0) circle [radius=2pt];
\filldraw (6,0) circle [radius=2pt];
\filldraw (7,0) circle [radius=2pt];
\filldraw (12,0) circle [radius=2pt];
\filldraw (2,19) circle [radius=2pt];
\filldraw (9,19) circle [radius=2pt];
\filldraw (11,19) circle [radius=2pt];
\filldraw (12,19) circle [radius=2pt];
\filldraw (13,19) circle [radius=2pt];
\filldraw (14,19) circle [radius=2pt];

  \draw[line width=1pt] (2,0)-- (2,19);

  \draw[line width=2pt,red] (4,0)--
  ++(0,2)--
  node[above]{$3$} ++(1,1)--
  node[above]{$4$} ++(1,1)--
  ++(0,2)--
  node[above]{$7$} ++(1,1)--
  node[above]{$8$} ++(1,1)--
  ++(0,7)--
  node[above]{$5^*$} ++(1,1)--
  (9,19);

  \draw[line width=2pt,red] (5,0)--
  ++(0,1)--
  node[above]{$2$} ++(1,1)--
  node[above]{$3$} ++(1,1)--
  ++(0,1)--
  node[above]{$5$} ++(1,1)--
  node[above]{$6$} ++(1,1)--
  node[below]{$7$} ++(1,1)--
  node[above]{$8$} ++(1,1)--
  node[above]{$9$} ++(1,1)--
  ++(0,5)--
  node[above]{$4^*$} ++(1,1)--
  node[above]{$5^*$} ++(1,1)--
  (14,19); 

  \draw[line width=2pt,dashed] (6,0)--
  ++(0,2)--
  node[below]{$3$} ++(1,1)--
  ++(0,4)--
  node[below]{$8$} ++(1,1)--
  node[above]{$9$} ++(1,1)--
  ++(0,5)--
  node[above]{$4^*$} ++(1,1)--
  ++(0,1)--
  node[above]{$6^*$} ++(1,1)--
  (11,19);

  \draw[line width=2pt,dashed] (7,0)--
  ++(0,2)--
  node[above]{$3$} ++(1,1)--
  ++(0,2)--
  node[above]{$6$} ++(1,1)--
  ++(0,1)--
  node[above]{$8$} ++(1,1)--
  ++(0,6)--
  node[above]{$4^*$} ++(1,1)--
  node[above]{$5^*$} ++(1,1)--
  node[above]{$6^*$} ++(1,1)--
  (13,19); 

  \draw[line width=1pt,blue] (12,0)--(12,19);

  \filldraw [green!70!black] (6,2) circle [radius=2pt];
  \filldraw [green!70!black] (7,3) circle [radius=2pt];
  \filldraw [green!70!black] (7,4) circle [radius=2pt];
  \filldraw [green!70!black] (7,7) circle [radius=2pt];
  \filldraw [green!70!black] (8,8) circle [radius=2pt];
  \draw [line width=1pt, blue] (8,8) circle [radius=4pt];
  \filldraw [green!70!black] (8,5) circle [radius=2pt];
  \filldraw [green!70!black] (9,6) circle [radius=2pt];
\end{tikzpicture}
\caption{The vertical tableau $\row_{\ge M_2+1}(T_3\sqcup \pd^{m_2}(U_2))\in
  \VT(\pi(\lambda)/\gamma)$ and the corresponding $n$-path $\mathbf{q}$. The chosen
  intersection $(a,b)$ is circled.}
   \label{fig:bad}
 \end{figure}
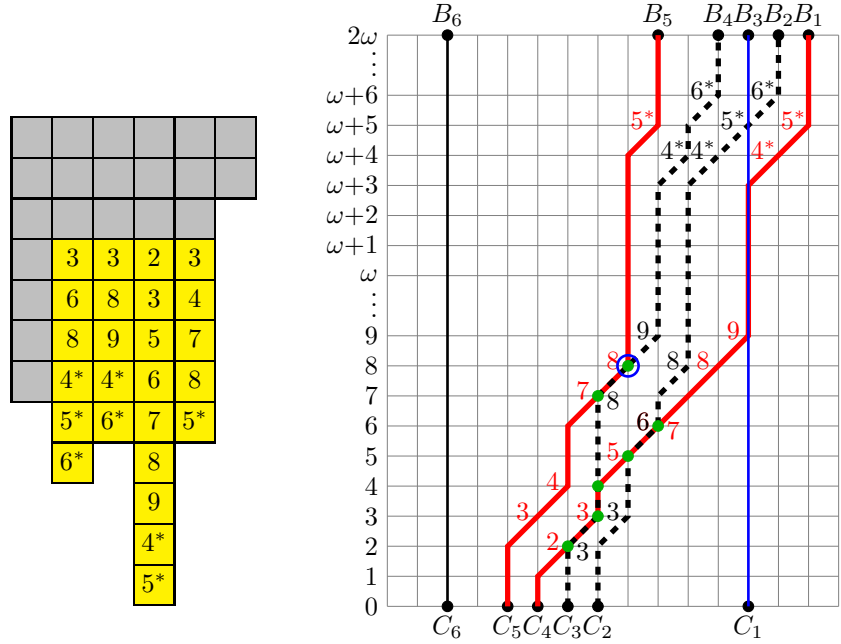

  In Step 3-1, we have $s=5$ and $\gamma=(6,6,5,1,1,1,1)$. The vertical tableau
  $\row_{\ge M_k+1}(T_{k}\sqcup \pd^{m_k}(U_{k-1})) \in\VT(\pi(\lambda)/\gamma)$ and
  its corresponding $n$-path $\mathbf{q}$ are shown in Figure~\ref{fig:bad}.

 \begin{figure}
   \centering
\begin{tikzpicture}[x=0.4cm, y=0.4cm]
\RSEgrid{15}96
\node[below] at (2,0){$C_6$};
\node[below] at (4,0){$C_5$};
\node[below] at (5,0){$C_4$};
\node[below] at (6,0){$C_3$};
\node[below] at (7,0){$C_2$};
\node[below] at (12,0){$C_1$};
\node[above] at (2,19){$B_6$};
\node[above] at (9,19){$B_5$};
\node[above] at (11,19){$B_4$};
\node[above] at (12,19){$B_3$};
\node[above] at (13,19){$B_2$};
\node[above] at (14,19){$B_1$};
\filldraw (2,0) circle [radius=2pt];
\filldraw (4,0) circle [radius=2pt];
\filldraw (5,0) circle [radius=2pt];
\filldraw (6,0) circle [radius=2pt];
\filldraw (7,0) circle [radius=2pt];
\filldraw (12,0) circle [radius=2pt];
\filldraw (2,19) circle [radius=2pt];
\filldraw (9,19) circle [radius=2pt];
\filldraw (11,19) circle [radius=2pt];
\filldraw (12,19) circle [radius=2pt];
\filldraw (13,19) circle [radius=2pt];
\filldraw (14,19) circle [radius=2pt];

  \draw[line width=1pt] (2,0)--
  (2,19);

  \draw[line width=2pt,red] (4,0)--
  ++(0,2)--
  node[above]{$3$} ++(1,1)--
  node[above]{$4$} ++(1,1)--
  ++(0,2)--
  node[above]{$7$} ++(1,1)--
  node[above]{$8$} ++(1,1)--
  node[above,black]{$9$} ++(1,1)--
  ++(0,5)--
  node[above,black]{$4^*$} ++(1,1)--
  ++(0,1)--
  node[above,black]{$6^*$} ++(1,1)--
  (11,19);

  \draw[line width=2pt,red] (5,0)--
  ++(0,1)--
  node[above]{$2$} ++(1,1)--
  node[above]{$3$} ++(1,1)--
  ++(0,1)--
  node[above]{$5$} ++(1,1)--
  node[above]{$6$} ++(1,1)--
  node[below]{$7$} ++(1,1)--
  node[above]{$8$} ++(1,1)--
  node[above]{$9$} ++(1,1)--
  ++(0,5)--
  node[above]{$4^*$} ++(1,1)--
  node[above]{$5^*$} ++(1,1)--
  (14,19); 

  \draw[line width=2pt,dashed] (6,0)--
  ++(0,2)--
  node[below]{$3$} ++(1,1)--
  ++(0,4)--
  node[below]{$8$} ++(1,1)--
  ++(0,7)--
  node[above]{$5^*$} ++(1,1)--
  (9,19);

  \draw[line width=2pt,dashed] (7,0)--
  ++(0,2)--
  node[above]{$3$} ++(1,1)--
  ++(0,2)--
  node[below]{$6$} ++(1,1)--
  ++(0,1)--
  node[above]{$8$} ++(1,1)--
  ++(0,6)--
  node[above]{$4^*$} ++(1,1)--
  node[above]{$5^*$} ++(1,1)--
  node[above]{$6^*$} ++(1,1)--
  (13,19); 

  \draw[line width=1pt,blue] (12,0)--(12,19);
  \draw [line width=1pt, blue] (8,8) circle [radius=4pt];
  \filldraw [green!70!black] (8,8) circle [radius=2pt];
\end{tikzpicture}\qquad
   \begin{ytableau}
    \muentry{}&  \muentry{}&\muentry{}&\muentry{}&\muentry{}&\muentry{}\\ 
    \muentry{}&  \muentry{}&\muentry{}&\muentry{}&\muentry{}&\muentry{}\\ 
    \muentry{}&  \muentry{}&\muentry{}&\muentry{}&\muentry{}\\ 
     \muentry{} & *(yellow)3 & *(yellow)3 & *(yellow)2 & *(yellow)3 \\
     \muentry{} & *(yellow)6 & *(yellow)8 & *(yellow)3 & *(yellow)4 \\
     \muentry{} & *(yellow)8 & *(yellow)5^* & *(yellow)5 & *(yellow)7 \\
     \muentry{} & *(yellow)4^* & \none & *(yellow)6 & *(yellow)8 \\
     \none & *(yellow)5^* & \none & *(yellow)7 & *(yellow)9\\
     \none & *(yellow)6^* & \none & *(yellow)8 & *(yellow)4^*\\
     \none & \none & \none & *(yellow)9 & *(yellow)6^*\\
     \none & \none & \none & *(yellow)4^* \\
     \none & \none & \none & *(yellow)5^* \\
     \none
   \end{ytableau} 
   \caption{The $n$-path $\mathbf{q}'$ and the corresponding vertical tableau
     $\Tab(\mathbf{q}')\in \VT(\pi'(\lambda)/\gamma)$, where
     $\pi'=\pi(4,5)$.}
   \label{fig:bad2}
 \end{figure}
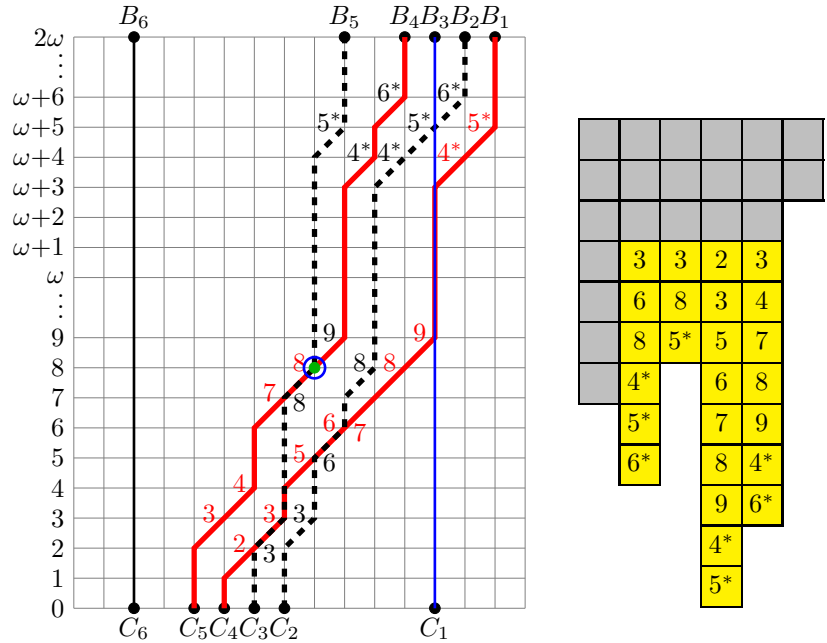

  In Step 3-2, there are 7 intersections among $\{q_i: d_{k+1}+1\le i\le
  s\}=\{q_2,q_3,q_4,q_5\}$ and the intersection $(a,b)=(8,8)$ of $q_i=q_3$ and
  $q_j=q_5$ is chosen. Then $\mathbf{q}'$ is obtained from $\mathbf{q}$ by
  exchanging the subpaths of $q_3$ and $q_5$ after $(8,8)$. See
  Figure~\ref{fig:bad2} for $\mathbf{q'}$ and its corresponding vertical
  tableau.

 \begin{figure}
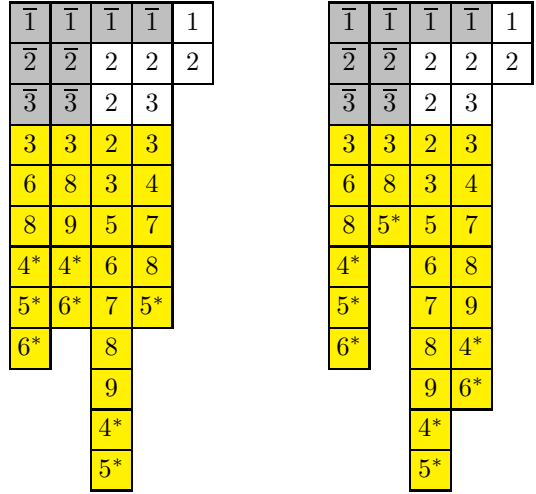

   \centering
   \begin{ytableau}
      \muentry1& \muentry1& \muentry1 & \muentry1 & 1 \\
      \muentry2 & \muentry2& 2 & 2 & 2\\
      \muentry3 & \muentry3 & 2 & 3\\
      *(yellow)3 & *(yellow)3 & *(yellow)2 & *(yellow)3 \\
      *(yellow)6 & *(yellow)8 & *(yellow)3 & *(yellow)4 \\
      *(yellow)8 & *(yellow)9 & *(yellow)5 & *(yellow)7 \\
      *(yellow)4^* & *(yellow)4^* & *(yellow)6 & *(yellow)8 \\
      *(yellow)5^* & *(yellow)6^* & *(yellow)7 & *(yellow)5^*\\
      *(yellow)6^* & \none & *(yellow)8 \\
      \none & \none & *(yellow)9 \\
      \none & \none & *(yellow)4^* \\
      \none & \none & *(yellow)5^* \\
   \end{ytableau} \qquad\qquad
   \begin{ytableau}
      \muentry1& \muentry1& \muentry1 & \muentry1 & 1 \\
      \muentry2 & \muentry2& 2 & 2 & 2\\
      \muentry3 & \muentry3 & 2 & 3\\
      *(yellow)3 & *(yellow)3 & *(yellow)2 & *(yellow)3 \\
      *(yellow)6 & *(yellow)8 & *(yellow)3 & *(yellow)4 \\
      *(yellow)8 & *(yellow)5^* & *(yellow)5 & *(yellow)7 \\
      *(yellow)4^* & \none & *(yellow)6 & *(yellow)8 \\
      *(yellow)5^* & \none & *(yellow)7 & *(yellow)9\\
      *(yellow)6^* & \none & *(yellow)8 & *(yellow)4^*\\
      \none & \none & *(yellow)9 & *(yellow)6^*\\
      \none & \none & *(yellow)4^* \\
      \none & \none & *(yellow)5^* \\
    \end{ytableau}
    \caption{The left diagram is
      $T_{k+1}\sqcup\pd^{m_k}(U_k)=T_3\sqcup\pd^{2}(U_2)$, where the part equal
      to $\Tab(\mathbf{q})$ is colored yellow. The right diagram is 
      $\widetilde{T}_{k+1}\sqcup\widetilde{U}_{k}
      =\widetilde{T}_{3}\sqcup\widetilde{U}_{2}$, where the part equal to
      $\Tab(\mathbf{q'})$ is colored yellow.}
   \label{fig:Phi2}
 \end{figure}

 \begin{figure}
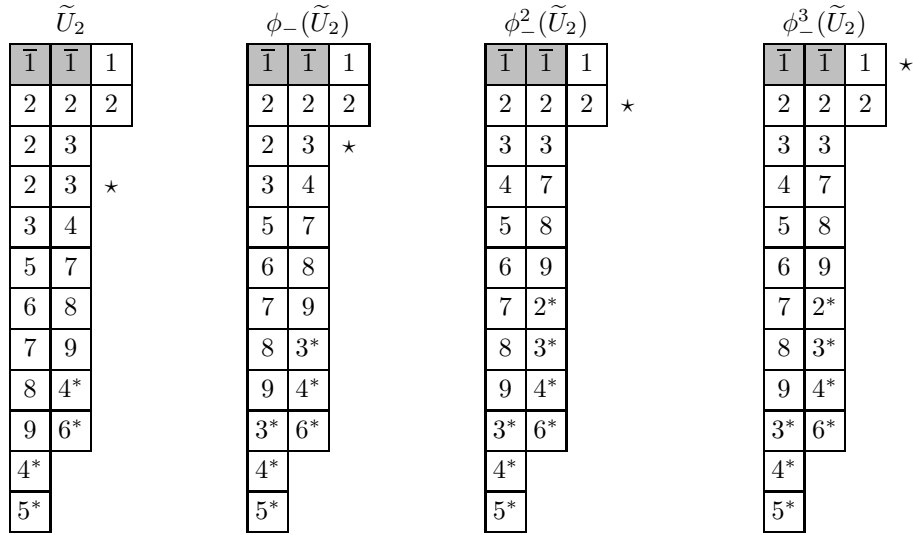

   \centering
   \begin{ytableau}
      \none & \none[\widetilde{U}_2]& \none\\
      \muentry1 & \muentry1 & 1 \\
      2 & 2 & 2\\
      2 & 3\\
      2 & 3 & \none[\star]\\
      3 & 4 \\
      5 & 7 \\
      6 & 8 \\
      7 & 9\\
      8 & 4^*\\
      9 & 6^*\\
      4^* \\
      5^* \\
   \end{ytableau} \qquad\qquad
   \begin{ytableau}
      \none & \none[\pu(\widetilde{U}_2)]& \none\\
      \muentry1 & \muentry1 & 1 \\
      2 & 2 & 2\\
      2 & 3 & \none[\star]\\
      3 & 4 \\
      5 & 7 \\
      6 & 8 \\
      7 & 9\\
      8 & 3^*\\
      9 & 4^*\\
      3^* & 6^*\\
      4^* \\
      5^* \\
   \end{ytableau} \qquad\qquad
   \begin{ytableau}
      \none & \none[\pu^2(\widetilde{U}_2)]& \none\\
      \muentry1 & \muentry1 & 1 \\
      2 & 2 & 2&\none[\star]\\
      3 & 3  \\
      4 & 7 \\
      5 & 8 \\
      6 & 9 \\
      7 & 2^*\\
      8 & 3^*\\
      9 & 4^*\\
      3^* & 6^*\\
      4^* \\
      5^* \\
   \end{ytableau} \qquad\qquad
   \begin{ytableau}
      \none & \none[\pu^3(\widetilde{U}_2)]& \none\\
      \muentry1 & \muentry1 & 1&\none[\star] \\
      2 & 2 & 2\\
      3 & 3  \\
      4 & 7 \\
      5 & 8 \\
      6 & 9 \\
      7 & 2^*\\
      8 & 3^*\\
      9 & 4^*\\
      3^* & 6^*\\
      4^* \\
      5^* \\
   \end{ytableau}
   \caption{The RSE-tableaux $\widetilde{U}_{k}=\widetilde{U}_2,
     \pu(\widetilde{U}_2), \pu^2(\widetilde{U}_2)$ and
     $\pu^3(\widetilde{U}_2)=\pu^{M_k}(\widetilde{U}_k)$ from left to right. The
     level of each RSE-tableau is marked by a star.}
   \label{fig:Phi3}
 \end{figure}

 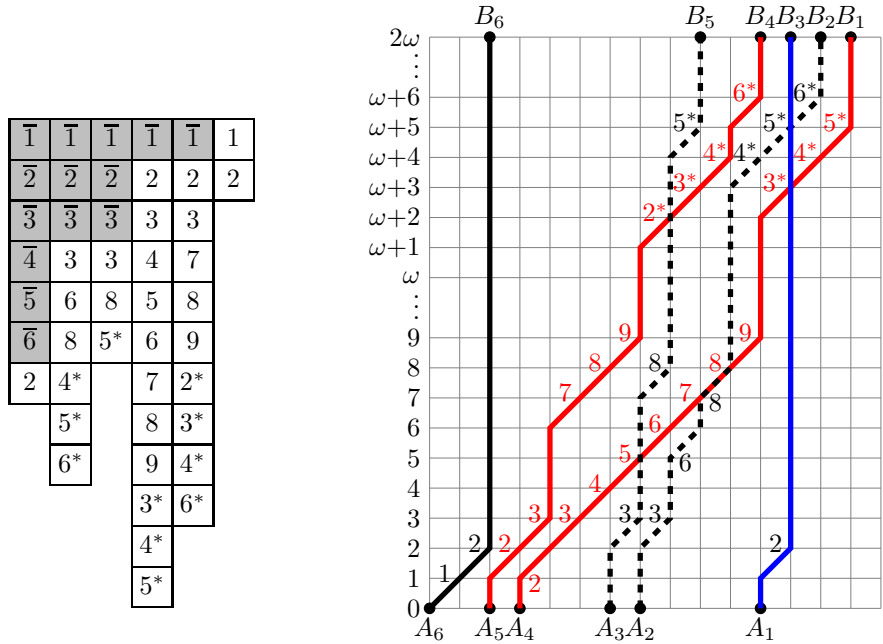
\begin{figure}
   \centering
   \begin{ytableau}
     \muentry1 & \muentry1& \muentry1& \muentry1 & \muentry1 & 1 \\
     \muentry2 & \muentry2 & \muentry2& 2 & 2 & 2\\
     \muentry3 & \muentry3 & \muentry3 & 3 & 3  \\
     \muentry4 & 3 & 3 & 4 & 7 \\
     \muentry5 & 6 & 8 & 5 & 8 \\
     \muentry6 & 8 & 5^* & 6 & 9 \\
     2 & 4^* & \none & 7 & 2^*\\
     \none & 5^* & \none & 8 & 3^*\\
     \none & 6^* & \none & 9 & 4^*\\
     \none & \none & \none & 3^* & 6^*\\
     \none & \none & \none & 4^* \\
     \none & \none & \none & 5^* \\
     \none
   \end{ytableau}\qquad\qquad
\begin{tikzpicture}[x=0.4cm, y=0.4cm]
\RSEgrid{15}96
\node[below] at (0,0){$A_6$};
\node[below] at (2,0){$A_5$};
\node[below] at (3,0){$A_4$};
\node[below] at (6,0){$A_3$};
\node[below] at (7,0){$A_2$};
\node[below] at (11,0){$A_1$};
\node[above] at (2,19){$B_6$};
\node[above] at (9,19){$B_5$};
\node[above] at (11,19){$B_4$};
\node[above] at (12,19){$B_3$};
\node[above] at (13,19){$B_2$};
\node[above] at (14,19){$B_1$};
\filldraw (0,0) circle [radius=2pt];
\filldraw (2,0) circle [radius=2pt];
\filldraw (3,0) circle [radius=2pt];
\filldraw (6,0) circle [radius=2pt];
\filldraw (7,0) circle [radius=2pt];
\filldraw (11,0) circle [radius=2pt];
\filldraw (2,19) circle [radius=2pt];
\filldraw (9,19) circle [radius=2pt];
\filldraw (11,19) circle [radius=2pt];
\filldraw (12,19) circle [radius=2pt];
\filldraw (13,19) circle [radius=2pt];
\filldraw (14,19) circle [radius=2pt];

  \draw[line width=2pt] (0,0)--
  node[above]{1} ++(1,1)--
  node[above]{2} ++(1,1)--
  (2,19);

  \draw[line width=2pt,red] (2,0)--
  ++(0,1)--
  node[above]{2} ++(1,1)--
  node[above]{3} ++(1,1)--
  ++(0,3)--
  node[above]{7} ++(1,1)--
  node[above]{8} ++(1,1)--
  node[above]{9} ++(1,1)--
  ++(0,3)--
  node[above]{$2^*$} ++(1,1)--
  node[above]{$3^*$} ++(1,1)--
  node[above]{$4^*$} ++(1,1)--
  ++(0,1)--
  node[above]{$6^*$} ++(1,1)--
  (11,19);

  \draw[line width=2pt, red] (3,0)--
  ++(0,1)--
  node[below]{2} ++(1,1)--
  node[above]{3} ++(1,1)--
  node[above]{4} ++(1,1)--
  node[above]{5} ++(1,1)--
  node[above]{6} ++(1,1)--
  node[above]{7} ++(1,1)--
  node[above]{8} ++(1,1)--
  node[above]{9} ++(1,1)--
  ++(0,4)--
  node[above]{$3^*$} ++(1,1)--
  node[above]{$4^*$} ++(1,1)--
  node[above]{$5^*$} ++(1,1)--
  (14,19);
  
  \draw[line width=2pt,dashed] (7,0)--
  ++(0,2)--
  node[above]{$3$} ++(1,1)--
  ++(0,2)--
  node[below]{$6$} ++(1,1)--
  ++(0,1)--
  node[below]{$8$} ++(1,1)--
  ++(0,6)--
  node[above]{$4^*$} ++(1,1)--
  node[above]{$5^*$} ++(1,1)--
  node[above]{$6^*$} ++(1,1)--
  (13,19);

  \draw[line width=2pt,dashed] (6,0)--
  ++(0,2)--
  node[above]{$3$} ++(1,1)--
  ++(0,4)--
  node[above]{$8$} ++(1,1)--
  ++(0,7)--
  node[above]{$5^*$} ++(1,1)--
  (9,19); 

  \draw[line width=2pt, blue] (11,0)--
  ++(0,1)--
  node[above,black]{$2$} ++(1,1)--
  (12,19);
\end{tikzpicture}
\caption{The tableau $T'=T_{r+1}\sqcup \dots\sqcup T_{k+2}\sqcup
  \widetilde{T}_{k+1} \sqcup \pd^{M_k}(\widetilde{U}_{k})= T_4\sqcup
  \widetilde{T}_{3}\sqcup \pd^{3}(\widetilde{U}_2)$ and the corresponding
  $n$-path $\pp'$.}
   \label{fig:Phi4}
 \end{figure}

  In Step 3-3, $\widetilde{T}_{k+1}\sqcup\widetilde{U}_{k}
  =\widetilde{T}_{3}\sqcup\widetilde{U}_{2}$ is obtained from
  $T_{k+1}\sqcup\pd^{m_k}(U_k)=T_3\sqcup\pd^{2}(U_2)$ by replacing the part
  equal to $\Tab(\mathbf{q})$ by $\Tab(\mathbf{q'})$, see Figure~\ref{fig:Phi2}.
  We then compute $\pu^{M_k}(\widetilde{U}_{k})=\pu^{3}(\widetilde{U}_{2})$ as
  in Figure~\ref{fig:Phi3}. Finally, the tableau $T'=T_{r+1}\sqcup \dots\sqcup
  T_{k+2}\sqcup \widetilde{T}_{k+1} \sqcup \pd^{M_k}(\widetilde{U}_{k})=
  T_4\sqcup \widetilde{T}_{3}\sqcup \pd^{3}(\widetilde{U}_2)$ and the
  corresponding $n$-path $\pp'$ are obtained as in Figure~\ref{fig:Phi4}.

\section*{Acknowledgments}
The author is grateful to Darij Grinberg for helpful discussions and useful
comments, and also for pointing out many typos. The author is also grateful to
the anonymous referees and U-Keun Song for useful comments. This work was done
while the author was participating the 2020 program in Algebraic and Enumerative
Combinatorics at Institut Mittag-Leffler. The author would like to thank the
institute for the hospitality and Sara Billey, Petter Br\"and\'en, Sylvie
Corteel, and Svante Linusson for organizing the program.

This material is based upon work supported by the Swedish Research
Council under grant no. 2016-06596 while the author was in residence at Institut
Mittag-Leffler in Djursholm, Sweden during the winter of 2020.

\end{document}